\documentclass[11pt]{amsart}
\usepackage[english]{babel}
\usepackage{amsmath}
\usepackage{amsfonts}
\usepackage{amssymb}
\usepackage{amsthm}
\usepackage{color}

\usepackage[english]{babel}
\usepackage{yfonts}
\usepackage[T1]{fontenc}
\usepackage[utf8x]{inputenc}
\usepackage{enumerate}
\usepackage{verbatim}
\usepackage{graphicx}
\usepackage{verbatim}
\usepackage{faktor}
\usepackage{xcolor}
\usepackage{xfrac}
\usepackage{tikz}
\usepackage[all]{xy}

\newcommand{\calV}{\mathcal V}
\newcommand{\LL}{\mathbb L}
\newcommand{\calG}{\mathcal G}
\renewcommand{\L}{\Lambda}
\newcommand{\PSL}{{\rm PSL}}
\renewcommand{\r}{\rho}
\newcommand{\Bb}{\mathcal B}
\newcommand{\deH}{\partial \mathbb H}
\renewcommand{\P}{{\rm P}}

\newcommand{\frakc}{\mathfrak{C}}
\newcommand{\fc}{\mathfrak{C}}

\newcommand{\ax}{{\rm ax}}
\newcommand{\YY}{\mathbb Y}
\newcommand{\pr}{{\rm pr}}
\newcommand{\ov}{\overline}

\newcommand{\wt}{\widetilde}
\newcommand{\Uu}{\mathcal U}
\newcommand{\Mm}{\mathcal M}
\newcommand{\s}{\sigma}

\newcommand{\Gr}{{\rm Gr}}
\newcommand{\Sym}{{\rm Sym}}

\newcommand{\End}{{\rm End}}

\newcommand{\os}{{\o,\sigma}}
\newcommand{\D}{\Delta}

\newcommand{\Yy}{\mathcal Y}
\newcommand{\g}{\gamma}

\renewcommand{\l}{\lambda}
\renewcommand{\s}{\sigma}
\renewcommand{\o}{\omega}
\newcommand{\G}{\Gamma}
\newcommand{\K}{\mathbb K}
\newcommand{\Id}{{\rm Id}}
\newcommand{\R}{\mathbb R}
\newcommand{\Z}{\mathbb Z}
\renewcommand{\P}{\mathbb P}
\renewcommand{\S}{\Sigma}
\renewcommand{\H}{\mathbb H}
\newcommand{\F}{\mathbb F}

\newcommand{\C}{\mathbb C}
\newcommand{\N}{\mathbb N}

\newcommand{\diag}{\mathrm{diag}}
\newcommand{\calC}{\mathcal C}
\newcommand{\SL}{{\rm SL}}
\renewcommand{\l}{\lambda}

\newcommand{\Aa}{\mathcal A}
\newcommand{\Ii}{\mathcal I}

\newcommand{\Gg}{\mathcal G}

\newcommand{\Oo}{\mathcal O}
\newcommand{\Xx}{\mathcal X}
\newcommand{\calP}{\mathcal P}
\newcommand{\calO}{\mathcal O}

\newcommand{\calL}{\mathcal L}
\newcommand{\Ll}{\mathcal L}
\newcommand{\Tt}{\mathcal T}
\newcommand{\calT}{\mathcal T}
\newcommand{\calB}{\mathcal B}
\newcommand{\tra}{\pitchfork}
\newcommand{\I}{\mathbb I}
\newcommand{\Qq}{\mathcal Q}

\newcommand{\Sp}{{\rm Sp}}
\newcommand{\GL}{{\rm GL}}

\newcommand{\<}{\langle}
\renewcommand{\>}{\rangle}

\newcommand{\bq}{\begin{equation}}
\newcommand{\eq}{\end{equation}}
\newcommand{\ba}{\begin{aligned}}
\newcommand{\ea}{\end{aligned}}
\newcommand{\be}{\begin{enumerate}}
\newcommand{\ee}{\end{enumerate}}
\newcommand{\bsm}{\left(\begin{smallmatrix}}
\newcommand{\esm}{\end{smallmatrix}\right)}                   
\newcommand{\bpm}{\begin{pmatrix}}
\newcommand{\epm}{\end{pmatrix}}
\newcommand{\barr}{\begin{displaymath}\begin{array}{cccc}}
\newcommand{\earr}{\end{array}\end{displaymath}}
\newcommand{\barrl}{\begin{displaymath}\begin{array}{lcl}}
\newcommand{\earrl}{\end{array}\end{displaymath}}
\newcommand{\barl}{\begin{displaymath}\begin{array}{l}}
\newcommand{\earl}{\end{array}\end{displaymath}}
\newcommand{\bxym}{ \begin{displaymath}\xymatrix }
\newcommand{\exym}{\end{displaymath}}

\newcommand{\tr}{{\rm tr}}

\theoremstyle{plain}
\newtheorem{thm}{Theorem}[section]
\newtheorem{lem}[thm]{Lemma}
\newtheorem{prop}[thm]{Proposition}
\newtheorem{cor}[thm]{Corollary}
\newtheorem*{teo*}{Theorem}

\newtheorem{claim}{Claim}

\theoremstyle{definition}
\newtheorem{example}[thm]{Example}

\newtheorem{defn}[thm]{Definition}
\newtheorem{remark}[thm]{Remark}

\newtheorem{rem}[thm]{Remark}

\newcommand{\thismonth}{\ifcase\month 
  \or January\or February\or March\or April\or May\or June%
  \or July\or August\or September\or October\or November%
  \or December\fi}

\begin{document}
\title[Maximal representations and buildings]{Maximal representations, non Archimedean Siegel spaces, and buildings}

\author[]{M. Burger}
\address{Department Mathematik, ETH Zentrum, 
R\"amistrasse 101, CH-8092 Z\"urich, Switzerland}
\email{burger@math.ethz.ch}

\author[]{M. B. Pozzetti}
\address{Mathematics Institute, Zeeman Building, University of Warwick, Coventry CV4 7AL – UK}
\email{B.Pozzetti@Warwick.ac.uk}

\date{\today}

\begin{abstract}
Let $\F$ be a real closed field. We define the notion of a maximal framing for a representation of the fundamental group of a surface with values in $\Sp(2n,\F)$. We show that ultralimits of maximal representations in $\Sp(2n,\R)$ admit such a framing, and that all maximal framed representations satisfy a suitable generalization of the classical Collar Lemma.  In particular this establishes a Collar Lemma for all maximal representations into $\Sp(2n,\R)$. We then describe a procedure to get from representations in $\Sp(2n,\F)$ interesting actions on affine buildings, and, in the case of representations admitting a maximal framing, we describe the structure of the elements of the group acting with zero translation length.
\end{abstract}
\maketitle

\section{Introduction}

Let $\S$ be a surface of genus $g$ with $p\geq 0$ punctures, and $V$ be a symplectic vector space over $\R$. A current theme in higher Teichm\"uller theory is to which extent classical hyperbolic geometry and some fundamental structures on the Teichm\"uller space of $\S$ carry over to the geometry and the moduli space of maximal representations of $\G=\pi_1(\S)$ into $\Sp(V)$. For instance compactifications of spaces of representations of $\G$ have been introduced and studied in \cite{APcomp,Alessandrini,Le}. In the context of Hitchin representations, asymptotic properties of diverging sequences are studied in \cite{Tengren1, Tengren2, Collier-Li, Loftin, AP3, KNPS, Mazzeo}.

The purpose of this paper is to study the action on an asymptotic cone of the symmetric space $\Xx$ associated to $\Sp(V)$ defined by a sequence $(\rho_k)_{k\in\N}$ of maximal representations $\rho_k:\G\to\Sp(V)$. More precisely, we fix a non principal ultrafilter $\o$ on $\N$ and let $(x_k)_{k\in\N}\in\Xx^\N$ be a sequence of basepoints. We say that a sequence of scales $(\l_k)_{k\in\N}$ is adapted to $(\rho_k,x_k)_{k\in\N}$ if 

$$\lim_\o\frac{D_S(\rho_k)(x_k)}{\l_k}<\infty$$
where for a representation $\rho$ and a generating set $S$ for $\G$ we define $D_S(\rho)(x)=\max_{\g\in S}d(\rho(\g)x,x)$. Observe that the above property is independent of the choice of the finite generating set $S$.

In this situation we obtain an action $^\o\!\rho_\l:\G\to {\rm Iso}(^\o\!\Xx_{\l})$ by isometries on the asymptotic cone $^\o\!\Xx_{\l}$ of the sequence $(\Xx,x_k,\frac d{\l_k})$. The space $^\o\!\Xx_{\l}$ is not only CAT(0)-complete, but, when the limit $\lim_\o \l_k$ is infinite, it is an affine building associated to the algebraic group $\Sp(V)$ over a specific field \cite{KleinerLeeb, APbuil, Thornton,  KT1}; more on this below. 
Depending on the choice of scales the representation $^\o\!\rho_\l$ might have a global fixed point, but, as it turns out, if the representations $\rho_k$ are maximal, the limiting action is always faithful. Our main result gives then the underlying geometric structure of the set of elements $\g$ in $\G$ whose translation length $L(^\o\!\rho_\l(\g))$ in $^\o\!\Xx_{\l}$ is zero; notice that  for an isometry of an affine building having zero translation length is equivalent to having a fixed point. 

For convenience we fix once and for all a complete hyperbolic metric on $\S$ of finite area, and identify $\G$ with a subgroup of $\PSL(2,\R)$.
In order to state this result we recall that a decomposition $\S=\bigcup_{v\in\calV}\S_v$ into subsurfaces with geodesic boundary gives rise to a presentation of $\G$ as fundamental group of a graph of groups with vertex set $\calV$ and vertex groups $\pi_1(\S_v)$. The group $\G$ acts on the associated Bass-Serre tree $\calT$ and in particular on its vertex set $\wt\calV$; observe that for $v\in\calV$ and $w\in\wt\calV$ lying above $v$, the stabilizer $\G_w$ of $w$ in $\G$ is isomorphic to $\pi_1(\S_v)$.
\begin{thm}\label{thm:1}
Let $\rho_k:\G\to\Sp(V)$ be a sequence of maximal representations, $(\l_k)_{k\geq 1}$ an adapted sequence of scales and $^\o\!\rho_{\l}$ the action of $\G$ on the asymptotic cone $^\o\!\Xx_{\l}$. Then $^\o\!\rho_\l$ is faithful. Moreover there is a decomposition $\S=\bigcup_{v\in \calV}\Sigma_v$ of $\S$ into subsurfaces with geodesic boundary such that
\begin{enumerate}
\item for every $\g\in\G$ whose corresponding closed geodesic is not contained in any subsurface,  $L(^\o\!\rho_\l(\g))>0$;
\item for every $v\in \calV$  there is the following dichotomy:
\begin{enumerate}
\item[(A)] for every $w\in\wt \calV$ lying above $v$, and any $\g\in\G_w$ which is not boundary parallel,  $L(^\o\!\rho_\l(\g))>0$,
\item[(B)]  for every $w\in\wt \calV$ lying above $v$, there is a point $b_w\in^\o\!\!\Xx_{\l}$ which is fixed by $\G_w$.
\end{enumerate}
\end{enumerate}
\end{thm}

An natural question is, given a sequence of maximal representations, how the choice of basepoints and scales influences the action of $\G$ on the asymptotic cone and in particular the decomposition given in Theorem \ref{thm:1}. Turning to this issue recall that for a maximal representation $\rho:\G\to \Sp(V)$ the displacement function $x\mapsto D_S(\rho(x))$ with respect to a generating set $S\subset \G$ achieves its minimum $\mu_S(\rho)$  in a compact region of the symmetric space $\Xx$. Given a sequence $(\rho_k)_{k\in\N}$ of maximal representations we have $\lim_\o \mu_S(\rho_k)<\infty$ if and only if, up to modifying the sequence on a set of $\o$-measure zero, $(\rho_k)_{k\in\N}$ is contained in a compact subset of the character variety of maximal representations. 

Assume thus $\lim_\o\mu_S(\rho_k)=\infty$. Then, choosing a basepoint $x_k\in\Xx$ such that $D_S(\rho_k)(x_k)=\mu_S(\rho_k)$, the sequence of scales $\mu_k:=\mu_S(\rho_k)$ is obviously adapted to the sequence $(\rho_k,x_k)$ and the resulting $\G$-action $^\o\!\rho_\mu$ on $^\o\!\Xx_\mu$ has no global fixed point. We show then (see Proposition \ref{prop:int2}) that if $y_k\in \Xx$ is a sequence of basepoints and $(\l_k)_{k\in\N}$ is an adapted sequence of scales such that $^\o\!\rho_\l$ has no global fixed point, then $^\o\!\Xx_\l$ equals $^\o\!\Xx_\mu$ with homothetic distance function and the actions $^\o\!\rho_\l$ and $^\o\!\rho_\mu$ coincide. In particular the decomposition of $\S$ into subsurfaces given by Theorem \ref{thm:1} is uniquely determined by the sequence $(\rho_k)_{k\in\N}$.

We say that a subsurface is of type A (resp. B) if the first (resp. the second) possibility in Theorem \ref{thm:1}(2) holds. One can show that any decomposition of the surface $\S$ and any assignment of type A or B to the subsurfaces can be realized by the limiting action for an appropriate sequence $(\rho_k)_{k\in\N}$.
 On the other hand  Theorem \ref{thm:1} suggests that in a generic limiting action without a global fixed point,  no element of $\G$ should have zero translation length. We plan on analyzing the properties of such representations in future work.

In case there is a subsurface of type B, the restriction $\rho_{w,k}:=\rho_k|_{\G_w}:\G_w\to \Sp(V)$ is a sequence of maximal representations to which the preceding discussion applies, that is either up to $\o$-measure zero the sequence is relatively compact in the character variety of $\G_w$ or there is an essentially unique choice of basepoints and scales such that the limiting action does not have a global fixed point. Since at each step the topological complexity of the surface decreases, this procedure stops after finitely many iterations and can be seen as an asymptotic expansion of the initial sequence $(\rho_k)_{k\in\N}$.

 When each subsurface in the decomposition of Theorem \ref{thm:1} is of type B, we can use the fixed points $b_w$ to construct a map from the Bass-Serre tree $\calT$ to  the asymptotic cone:
\begin{thm}\label{thm:2}
Assume that { for any subsurface of the decomposition possibility B holds.}
Then there is a $^\o\!\rho_\l$-equivariant quasi isometric embedding $\calT\to ^\o\!\Xx_{\l}$.
\end{thm}
 In the case of a vector space of dimension 2, maximal representations correspond to holonomies of hyperbolizations; in this case the second possibility in Theorem \ref{thm:1} occurs for example for sequences of hyperbolizations obtained by pinching a multicurve. In this case the image of the quasi isometric embedding of Theorem \ref{thm:2} is a simplicial subtree of the asymptotic cone $^\o\!\Xx_\l$. In higher rank it is possible to construct examples in which the image of the Bass-Serre tree is not totally geodesic in the affine building $^\o\Xx_\l$.

 We finish our discussion about ultralimits of maximal representations mentioning two interesting geometric properties of maximal representations that can be deduced from our work. Let $S$  be a connected generating set, namely a generating set for $\G$ such that the union of the closed geodesics representing the elements of $S$ is a connected subset of $\S$ and  let $L_S(\rho)$ denote the maximal displacement of an element in the generating set S:
$$L_S(\rho)=\max_{\g\in S}L(\rho(\g)).$$
\begin{cor}\label{cor:intro1}
 Let $S\subset \G$ be a connected generating set for $\G$. Then there is a constant $C$ depending only on $S$ and $2n=\dim V$ such that for any maximal representation $\rho:\G\to\Sp(V)$ we have
 $$(\ln 2)\sqrt n\leq L_S(\rho)\leq \mu_S(\rho)\leq C L_S(\rho).$$
\end{cor}

We say that two diverging sequences of real numbers $(\l_k)$ and $(\mu_k)$ have the same growth rate according to the ultrafilter $\o$ if $\lim_\o \l_k/\mu_k$ is finite and non-zero. 
\begin{cor}\label{cor:intro2}
 Let $(\rho_k)_{k\in\N}$ be a sequence of maximal representations of the fundamental group $\G$ of a surface of genus $g$ with $p$ punctures. Then, varying $\g\in\G$  there are at most  $8g-8+4p$ distinct growth rate classes among the sequences $L(\rho_k(\g))_{k\in\N}$.  
\end{cor}

\subsection{Real closed fields}
The building structure on $^\o\!\Xx_{\l}$ alluded to previously comes about as follows. Assume that the sequence of scales $\l_k$ is unbounded. Then $\s=(e^{-\l_k})$ is an infinitesimal in the field $\R_\o$ of the hyperreals and the building $^\o\!\Xx_{\l}$ is associated to $\Sp(V\otimes \R_{\o,\s})$ where $\R_{\o,\s}$ is the valuation field introduced by Robinson \cite{ APbuil, Thornton}.
The characterizing properties of the representations arising as ultralimits of maximal representations make sense in the more general context of symplectic groups over arbitrary real closed fields. 
Indeed when $V_\F$ is a symplectic vector space over a real closed field $\F$, the Kashiwara cocycle classifies the orbits of $\Sp(V_\F)$ on triples of pairwise transverse Lagrangians and can be used to select maximal triples (see Section \ref{sec:2.2} for a precise definition of maximal triples).
The general object of our study are representations which admit a maximal framing:
\begin{defn}
A representation $\rho:\G\to \Sp(V_\F)$ admits a maximal framing if there exists an equivariant map $\phi:S\to \Ll(V_\F)$ from a $\G$-invariant subset $S\subset \partial \H^2$ including the fixed points of hyperbolic elements of $\G$, into the space $\Ll(V_\F)$ of Lagrangians, such that for every positively oriented triple $(x,y,z)$ in $S$, the image $(\phi(x),\phi(y),\phi(z))$ is maximal.
\end{defn}
If $\F=\R$ any maximal representation admits a maximal framing \cite[Theorem 8]{BIW}, and we show in Corollary \ref{cor:10.4} that this is also true for all ultralimits of maximal representations.  Even more, the class of representations admitting a maximal framing is closed under the natural reduction process we are now going to describe. Let $\Oo\subset\F$ be an order convex local subring, its quotient by the maximal ideal, denoted $\F_{_\Oo}$, is real closed as well. Assume now that there exists a symplectic basis of $V_\F$ such that $\rho(\G)\subset \Sp(2n,\Oo)$.  We can then consider the composition $\rho_{_\Oo}$ of $\rho$ with the quotient homomorphism $\Sp(2n,\Oo)\to \Sp(2n,\F_{_\Oo})$:
\begin{thm}\label{thm:3}
Assume that $\rho:\G\to \Sp(V_\F)$ admits a maximal framing, then the reduction $\rho_{_{_\Oo}}:\G\to \Sp(V_{\F_{_{_\Oo}}})$ admits a maximal framing as well.
\end{thm}

Theorem \ref{thm:3} allows in general to obtain  well controlled actions on affine buildings. Indeed, for each infinitesimal $\s>0$, the set of elements of $\F$ which are comparable with $\s$:
$$\Oo_\s=\{x\in\F|\; |x|\leq \s^{-k} \text{ for some $k\in \Z$}\}$$
forms an order convex subring of $\F$. We denote by $\F_\s$ its residue field which inherits from $\Oo_\s$ an order compatible valuation. As a consequence, to any reductive algebraic group over $\F_\s$ is associated an affine Bruhat-Tits building \cite{BT}. Since $\G$ is finitely generated, for each representation $\rho:\G\to\Sp(V_\F)$ and every choice of a basis, it is possible to chose an infinitesimal $\s$ such that $\rho(\G)\subset\Sp(2n,\Oo_\s)$. By passing to the quotient $\rho_\s:\G\to \Sp(2n,\F_\s)$ we get an action on the affine building associated to $\Sp(2n,\F_\s)$.
The main result for maximal framed representations over real closed fields with valuation is
\begin{thm}\label{thm:6}
Let $\rho:\G\to \Sp(V_\LL)$ be a maximal framed representation, where $\LL$ is real closed with order compatible valuation, and let $\Bb$ be the Bruhat-Tits affine building associated to $\Sp(V_\LL)$. Then the action of $\G$ on $\Bb$ satisfies the conclusions of  Theorem \ref{thm:1}.
\end{thm}
When $\LL$ is a real closed field with order compatible valuation, we denote by $\Uu$ the order convex  valuation ring with residue field $\LL_\Uu$. We already mentioned that the action on the affine building associated to a representation $\rho:\G\to \Sp(V_\LL)$ might have a global fixed point. However, when this is the case, it is possible to find a symplectic basis of $V_\LL$ such that $\rho(\G)\subset \Sp(2n,\Uu)$ and, if $\rho$ admits a maximal framing, then it follows from Theorem \ref{thm:3} that the reduction $\rho_\Uu:\G\to \Sp(2n,\LL_\Uu)$ has the same property. In particular this can be used to study the restriction of the representation $\rho$ to the subsurfaces defined in Theorem \ref{thm:6}.

As a consequence of Theorem \ref{thm:6} we get a concrete way of checking if a representation $\rho$ admitting a maximal framing has a global fixedpoint: if $S$  is a connected generating set for $\G$,  then $\rho$ has a global fixed point if and only if each element of $S$ has a fixed point (see Corollary \ref{cor:3} for a precise formulation of this result and some further comments).

\subsection{Tools} 
We now turn to a short description of the key tools we develop in this paper. In the context of his approach to the compactification of the Teichm\"uller space \cite{Brum2}, Brumfiel studied non-Archimedean hyperbolic planes \cite{Brumfiel}: for any ordered field $\F$, he associates to $\PSL_2(\F)$ a non-standard hyperbolic plane $\H\F^2$ and, for fields with valuation, he introduces a pseudodistance on $\H\F^2$ whose Hausdorff quotient is the $\R$-tree associated to $\PSL_2(\F)$.
Inspired by Brumfiel's work (cfr. also \cite{KT1}), we associate to a symplectic group $\Sp(2n,\F)$ over a real closed field $\F$ the space 
$$\Xx_{\F}=\{X+iY| X,Y \in\Sym(n,\F), Y \text{ positive definite}\}$$
on which $\Sp(2n,\F)$ acts by fractional linear transformations. The $\Sp(2n,\F)$-space $\Xx_{\F}$ can be thought of as a non-standard version of the Siegel upper half-space. Using a matrix valued crossratio we define, for any two transverse Lagrangians $a,b\in\Ll(\F^{2n})$, the \emph{ $\F$-tube} $\Yy_{a,b}$ which is the non-standard symmetric space associated to the stabilizer in $\Sp(2n,\F)$ of the pair $(a,b)$, a group isomorphic to $\GL(n,\F)$. In the case of the hyperbolic plane, the $\F$-tubes are just the Euclidean half-circles joining the ideal points $a,b$. Given a representation $\rho:\G\to \Sp(2n,\F)$ admitting a maximal framing $\phi:S\to\Ll(\F^{2n})$ we can associate to every hyperbolic element $\g\in\G$ the $\F$-tube $\Yy_\g=\Yy_{\phi(\g^-),\phi(\g^+)}$ where $\{\g^-,\g^+\}$ are the fixed points of $\g$ in $\partial\H^2$. One key property that we exploit is that the intersection pattern of the axis of hyperbolic elements in $\G$ is reflected in the intersection pattern of the corresponding $\F$-
tubes.
When the field $\F$ has an order compatible valuation, there is a natural $\R_{\geq 0}$-valued pseudodistance on $\Xx_{\F}$, and the relation between crossratios and this pseudodistance allows us to quantify the intersection pattern of the $\F$-tubes. Finally we exploit that the Hausdorff quotient of $\Xx_{\F}$ can be identified with the set of vertices of the affine Bruhat-Tits building associated to $\Sp(2n,\F)$.

\subsection{Collar Lemma}
We finish this introduction discussing another geometric property of representations admitting a  maximal framing, which is at the basis of most of the results we discussed so far. Recall that, since any element $g\in\Sp(V)$ is conjugate to $^t\!g^{-1}$, the set of eigenvalues of a symplectic element is closed with respect to inverses: if $\l$ is an eigenvalue of $g$, the same is true for $\l^{-1}$. With a slight abuse of terminology we say that two hyperbolic elements $\g,\eta\in\G<\PSL_2(\R)$ intersect if their axis do. 
\begin{thm}[Collar Lemma]\label{lem:collar_i}
Let $\F$ be a real closed field and $\rho:\G\to\Sp(V_\F)$ be a representation admitting a maximal framing. Then if $\g\in\G$ is hyperbolic, $\rho(\g)$ has no eigenvalue of absolute value 1. 
Let $|\l_1(\g)|\geq\ldots\geq |\l_n(\g)|> 1$ be  the eigenvalues of absolute value larger than 1.  If the hyperbolic elements $\g,\eta$ in $\G$ have positive intersection number, then:
\begin{eqnarray}
& {\displaystyle |\l_1(\g)|^{2n}\geq\frac 1{|\l_n(\eta)|^{2}-1}}.\\
&{\displaystyle \left(\prod_{i=1}^n\left|\l_i(\g)\right|^{2/n}-1\right)\left(\prod_{i=1}^n\left|\l_i(\eta)\right|^{2/n}-1\right)\geq 1}.
\end{eqnarray}
\end{thm}
Here $|\cdot|$ denotes the $\F$-valued absolute value on $\F[i]$, and we count the eigenvalues with their multiplicity as roots of the characteristic polynomial.
 
We now draw some consequences in the case of classical maximal representations (see Corollary \ref{cor:7.5} for similar considerations in the case of actions on buildings).
It was established by Siegel (cfr. \cite[Theorem 3]{Siegel}) that, under suitable normalizations, the translation length of an isometry $g\in \Sp(2n,\R)$ on the symmetric space $\Xx_\R$ is 
$$L(g)=2\sqrt{\sum_{i=1}^n \ln^2(|\l_i(g)|)}.$$
Using this formula we get, from Theorem \ref{lem:collar_i} (2), the following  
\begin{cor}
Let $\rho:\G\to\Sp(2n,\R)$ be a maximal representation. If $\g$ and $\eta$ have positive intersection number, then
$${ \left(e^{\frac{ L(\rho(\g))}{\sqrt n}}-1\right)\left(e^{\frac{L(\rho(\eta))}{\sqrt n}}-1\right)\geq 1}.$$
\end{cor}
Using that $e^x-1\leq 2x$ for $x\leq 1$, we get that, if $L(\rho(\eta))\leq \sqrt n$, then
$$\frac {L(\rho(\g))}{\sqrt n}\geq \ln\left(\frac{\sqrt n}{2 L(\rho(\eta))}\right)$$
which exhibits the same asymptotic growth relation as in the Teichm\"uller setting. However it is worth remarking that, as opposed to the classical Collar Lemma, Theorem \ref{lem:collar_i} is not just a consequence of the Margulis Lemma: in our setting the sets of minimal displacement of the isometries $\rho(\g)$ and $\rho(\eta)$ do not necessarily intersect.
A similar version of the Collar Lemma in the framework of Hitchin representations has been recently established in \cite{collar} (see Remark \ref{rem:collar} for a comparison with our results).
\subsection{Outline of the paper}
In Section \ref{sec:1} we define three different models for the nonstandard symmetric space and we study the action of $\Sp(V)$ on $n$-tuples of transverse Lagrangians. Section \ref{sec:collar} is devoted to the proof of the Collar Lemma, Theorem \ref{lem:collar}, for representations admitting a maximal framing. The matrix valued crossratio and the $\F$-tubes are introduced and studied in Section \ref{sec:3}. In Section \ref{sec:4} we focus on order convex subrings and describe how to obtain representations over the residue field. The main result of the section is Theorem \ref{thm:bdrymap} (Theorem \ref{thm:3} in the introduction), whose proof also exploits the geometric input coming from the Collar Lemma. In Section \ref{sec:building} we restrict to fields with valuations and use the crossratio to describe the projection from the nonstandard symmetric space to the affine Bruhat-Tits building. In  Section \ref{sec:fixedpoint} we initiate our study of elements with zero translation length: to each such 
element we associate a pair of canonical fixed points (Proposition \ref{prop:lazy}) and give sufficient conditions for these points to coincide (Proposition \ref{prop:intersectinglazy}). The proof of the Decomposition Theorem \ref{thm:6} (Theorem \ref{thm:dec}) occupies Section \ref{sec:dec}, while Theorem \ref{thm:2} is proven in Section \ref{sec:qi}. In the last section of the paper we discuss the relation between ultralimits of maximal representations and representations in symplectic groups over the Robinson field $\R_{\os}$. This allows us to deduce Theorem \ref{thm:1} from the more general Theorem \ref{thm:6}, and, in the case of closed surfaces, to completely characterize representations in $\Sp(2n,\R_\os)$ which admit a maximal framing (Theorem \ref{thm:10.5}).

\subsection{Acknowledgments}
The major part of this work was done during the MSRI program "Dynamics on Moduli Spaces of Geometric structures" 2015; we thank the organizers and the participants of this program who have created an intellectually stimulating atmosphere.  We thank in particular Anne Parreau for important discussions and more specifically for encouraging us to prove Theorem \ref{thm:1}, Domingo Toledo for his helpful comments and Thomas Huber for providing a proof of Proposition \ref{prop:appendix} valid in the general context of real closed fields. The first named author thanks the Clay foundation and the Simons foundation for support. This material is based upon work supported by the National Science Foundation under Grant No. 0932078 000 while the authors were in residence at  MSRI.
\section{Symplectic geometry over real closed fields}\label{sec:1}

\subsection{Basic objects}\label{sec:basic}
Let $V$ be a $2n$-dimensional vector space over a field $\F$, endowed with a  symplectic form $\<\cdot,\cdot\>$. 
We denote by $\Sp(V)$ the symplectic group, the subgroup of $\GL(V)$ preserving the form $\langle\cdot,\cdot \rangle$, and by $\Ll(V)\subset \Gr_{n}(V)$  the set of  Lagrangian subspaces:  the set of maximal isotropic subspaces of $V$. Whenever a Lagrangian $l$ is  fixed we denote by $\Ll(V)^l$ the set of Lagrangians transverse to $l$, and by $\Qq(l)$ the vector space of quadratic forms on $l$.

Given $a,b$ in $\Ll(V)$ transverse, we recall the construction of an affine chart
$$j_{a,b}:\Qq(a)\to \Ll(V)^{b}.$$

For each element $f$ in $\Qq(a)$ we denote by $b_f:a\times a\to \F$ the associated symmetric bilinear form. Since $a$ and $b$ are transverse, the  symplectic pairing induces an isomorphism of $a$ with the dual of $b$. We denote by $T_f:a\to b$ the unique linear map satisfying $$\<v,T_f(w)\>=b_f(v,w), \quad \text{ for $v,w\in a$}$$ 
The subspace of $V$ defined by
$$j_{a,b}(f):=\{v+T_f(v)|v\in a\}$$
is a Lagrangian subspace transverse to $b$.

Conversely if $l$ is transverse to $b$, any vector $v$ in $a$ can be written uniquely as a combination of a vector in $b$ and a vector in $l$. This allows us to define a linear map $T_{a,b}^{l}:a\to b$ by requiring that $v+T_{a,b}^{l}(v)\in l$. In turn we can use $T^{l}_{a,b}$ to define the quadratic form $Q_{a,l,b}$ on $a$:
$$Q_{a,l,b}(v)=\<v,T^{l}_{a,b}(v)\>,\quad v\in a$$
which satisfies $j_{a,b}(Q_{a,l,b})=l$. When $a$ and $b$ are clear from the context we will often abbreviate $Q_{a,l,b}$ by $Q_l$.

\subsection{Three models of the Siegel space}\label{sec:models}

The symmetric space associated to the symplectic group $\Sp(2n,\R)$ was extensively studied by Siegel \cite{Siegel} and is often referred to as the Siegel space.
We now show that the three most studied models for the Siegel space can be defined over arbitrary ordered fields, are always equivariantly isomorphic, and give rise to interesting geometries.

We fix an ordered field $\F$. Clearly the polynomial $f(x)=x^2+1$ is irreducible in $\F[x]$, we denote by $i\in \ov \F$  a root of the polynomial $f$ and  by $\K$ the splitting field of $f$, the degree two extension $\K=\F[i].$  If $V$ is a $2n$-dimensional vector space over $\F$ endowed with a  symplectic form $\<\cdot,\cdot\>$, we denote by $V_\K$ the ``complexification'' $V_\K=V\otimes \K$ and by $\langle\cdot,\cdot \rangle_\K:V_\K^2\to \K$ the $\K$-linear extension to $V_\K$ of  $\<\cdot,\cdot\>$.

The first model of the Siegel space consists of the set of \emph{compatible complex structures} on $V$, that is
$$\mathbb X_V=\{J\in \GL(V)|\; J^2=-\Id, \<J\cdot,\cdot\> \text{ is a scalar product}\}.$$ 
The set $\mathbb X_V$ is a semialgebraic subset of $\End(V)$ on which the symplectic group $\Sp(V)$ acts  by conjugation. For $J\in \mathbb X_V$ we will denote by $(\cdot,\cdot)_J:=\<J\cdot,\cdot\>$ the corresponding scalar product.

The second model of the Siegel space corresponds to the image of the \emph{Borel embedding} (cfr. \cite[Section 2.1.1]{BILW}, \cite{Satake}). As in the real case we realize $\mathbb X_V$ as a semialgebraic subset $\Tt_V$ of $\Ll(V_\K)$. Indeed if $J\in \GL(V)$ is an element of $\mathbb X_V$ the complexification $J\otimes \mathbb{I}_\K$ is diagonalizable over $\K$. It is easy to verify that the eigenspaces $L_J^\pm$  of $J\otimes\mathbb I_\K$ with respect to the eigenvalues $\pm i$ are elements  of $\Ll(V_\K)$. 
If we denote by $\sigma:V_\K\to V_\K$ the complex conjugation with respect to the real form $V$, we get that $\sigma (L_J^{\pm})=L_J^\mp$. The image $\Tt_V$ of the Borel embedding can be characterized as the set
$$\Tt_V=\{L\in \Ll(V_\K)|\;i\langle\cdot,\sigma(\cdot)\rangle_\K|_{L\times L} \text{ is positive definite}\}.$$
\begin{lem}
The algebraic map
$$\begin{array}{ccc}
   \mathbb X_V &\to&\Tt_V\\
   J&\mapsto& L^+_J
  \end{array}
$$
induces an $\Sp(V)$-equivariant bijection.
\end{lem}
\begin{proof}
If $v=x+iy$ is an eigenvector for the endomorphism $J\otimes \I_\K$ of eigenvalue $i$, it follows that $y=-Jx$. In particular the restriction of $i\langle\cdot,\sigma(\cdot)\rangle_\K$ satisfies 
$$\begin{array}{rl}
i\langle v,\sigma(v)\rangle_\K&=i\langle x,iJx\rangle_\K+i\langle -iJx,x\rangle_\K\\
&=2\langle Jx,x\rangle_\K
\end{array}$$
and this implies that the image of $\mathbb X_V$ is contained in $\Tt_V$.

Conversely if $L\in\Ll(V_\K)$ is such that  $i\langle\cdot,\sigma(\cdot)\rangle_\K|_{L\times L}$ is positive definite, $L$ is transverse to $\sigma (L)$ since the restriction of the aforementioned Hermitian form to this second subspace is negative definite. We denote by $J_L$ the endomorphism of $V_\K$ defined by imposing that $J_L(v)=iv$ for each $v$ in $L$ and $J_L(v')=-iv'$ for each $v'\in \sigma(L)$. 

Since any element $w$ of $V$ can be written uniquely as $w=v+\sigma(v)$ for some $v\in L$, and in particular $Jw=iv-i\sigma(v)=iv+\sigma(iv)\in V$, the endomorphism $J_L$ preserves the real structure $V$. Moreover since the Hermitian form  $i\langle\cdot,\sigma(\cdot)\rangle_\K|_{L\times L}$ is by assumption positive definite, the quadratic form $\langle J\cdot,\cdot\rangle$ is positive definite.
\end{proof}

The third and most concrete model for the Siegel space is the \emph{upper half-space} $\Xx_\F$, a specific set of $\K$-valued symmetric matrices:  
$$\Xx_\F=\{X+iY|\;X\in {\rm Sym} (n,\F),\; Y\in {\rm Sym}^+(n,\F)\}.$$
Here ${\rm Sym}(n,\F)$ denotes the vector space of symmetric $n$-dimensional matrices with coefficients in $\F$ and  ${\rm Sym}^+(n,\F)$ denotes the properly convex cone in ${\rm Sym}(n,\F)$ consisting of positive definite symmetric matrices. 

In order to establish a bijection between $\Tt_V$ and $\Xx_\F$, we fix a Lagrangian $l_\infty$ in $\Ll(V)$, a complex structure $J\in\mathbb X_\F$ and a basis $e_1,\ldots,e_n$ of $l_\infty$ which is orthonormal for $(\cdot,\cdot)_J$. The matrix representing the symplectic form with respect to the basis 
$$\Bb=\{e_1,\ldots,e_n,-Je_1,\ldots,-Je_{n}\}$$ 
of $V$ is $\bsm 0&-\Id\\\Id&0\esm$. Moreover using the basis $\Bb$ we can associate to any $2n\times n$ dimensional matrix $M$ of maximal rank the $n$-dimensional subspace of $V$ spanned by the columns of $M$. We use this to give an explicit identification of $\Sym(n,\K)$ with the affine chart of $\Ll(V_\K)$ which consists of subspaces transverse to $l_\infty$:
$$\begin{array}{cccc}
   \iota:{\rm Sym}(n,\K)&\to&\Ll(V_\K)\\
   X&\mapsto&\bsm X\\\Id\esm.
  \end{array}
$$
It is easy to verify that if we use the basis $\{-Je_1,\ldots,-Je_n\}$ to identify the space ${\rm Sym}(n,\K)$ with $\Qq(Jl_\infty)$, we get that the map $\iota$ corresponds to the map $j_{Jl_\infty,l_\infty}$ described in Section \ref{sec:basic}.

Since the restriction of $i\langle\cdot,\sigma(\cdot)\rangle_\K$ to $l_\infty$ is identically zero, every element $l$ of $\Tt_V$ belongs to the image of $\iota$, and it is easy to verify that the restriction of $i\langle\cdot,\sigma(\cdot)\rangle_\K$ to $\iota(X+iY)$ can be represented by the matrix $2Y$. In particular
$\iota$ restricts to a bijection between $\Tt_V$ and $\Xx_\F$.
Notice that the restriction of $\iota$ to the subset of $\F$-valued symmetric matrices has image in $\Ll(V)$ and gives a parametrization of the affine chart of $\Ll(V)$ consisting of Lagrangians transverse to $l_\infty$.

It follows from the identification between $\mathbb X_V$ and $\Xx_\F$ that the symplectic group $\Sp(2n,\F)$ acts on $\Xx_\F$ by fractional linear transformations:
$$\bpm A&B\\C&D\epm\cdot Z=(AZ+B)(CZ+D)^{-1}.$$

It will be useful in the following to record that, with our choice for the symplectic form, an element $\bsm A&B\\C&D\esm$ belongs to $\Sp(2n,\F)$ if and only if
$$\left\{\begin{array}{l}
   ^t\!AD-^t\!CB=\Id\\
   ^t\!AC=C^t\!A\\
   ^t\!BD=^t\!DB.
  \end{array}\right.
$$
In order to achieve transitivity of the symplectic group on the Siegel upper half-space, we need to restrict to real closed fields:
\begin{defn}
 A \emph{real closed field} is an ordered field $\F$ in which every positive element is a square and such that every polynomial in one variable over $\F$ factors into linear and quadratic factors.
\end{defn}
\begin{lem}\label{lem:2.3}
If $\F$ is a real closed field, the symplectic group $\Sp(2n,\F)$ acts transitively on $\Xx_\F$. 
\end{lem}
\begin{proof}
Since $\F$ is, by assumption, real closed, every symmetric matrix is diagonalizable by an orthogonal matrix, and, as soon as it is positive definite, it admits an unique positive square root \cite[Section 2-4]{Kap}.
Let now $X+iY$ be a point in $\Xx_\F$ and let $S$ be the square root of $Y$. We have
$$X+iY=\bpm S&XS^{-1}\\0&S^{-1}\epm\cdot i\Id.$$
\end{proof}

\subsection{Action on $\F$-Lagrangians}\label{sec:2.2}
We now want to understand the action of $\Sp(V)$ on $n$-tuples of pairwise transverse Lagrangians. We denote this set by $\Ll(V)^{(n)}$:
$$\Ll(V)^{(n)}=\{(l_1,\ldots,l_n)\in\Ll(V)^{n}|\;l_i\tra l_j\}.$$

It is a general fact that, for any field $\F$, the symplectic group acts transitively on pairs of transverse Lagrangians. 
\begin{lem}
 The symplectic group $\Sp(V)$ acts transitively on $\Ll(V)^{(2)}$.
\end{lem}
Recall from Section \ref{sec:basic} that, whenever two transverse Lagrangians $a,b$ are fixed, we have an identification $j_{a,b}:\Qq(a)\cong\Ll(V)^{b}$, and we denote by $Q_{a,l,b}$ the inverse image $j_{a,b}^{-1}(l)$. Clearly for any element $g$ in $\Sp(V)$ the quadratic forms $Q_{l_1,l_2,l_3}$ and $Q_{gl_1,gl_2,gl_3}$ are equivalent. As it turns out, the equivalence class of the quadratic form $Q_{l_1,l_2,l_3}$ is a complete invariant of the triple $(l_1,l_2,l_3)$ up to the symplectic group action: 
\begin{prop}
 The triples $(l_1,l_2,l_3), (m_1,m_2,m_3)$ in $\Ll(V)^{(3)}$ are equivalent modulo the symplectic group action if and only if the quadratic forms $Q_{l_1,l_2,l_3}$ and $Q_{m_1,m_2,m_3}$ are equivalent.
\end{prop}
\begin{proof}
Since $\Sp(V)$ is transitive on pairs of transverse subspace we can assume that $l_1=m_1=a$ and $l_3=m_3=b$. The result now follows from the fact that the stabilizer in $\Sp(V)$ of the pair $a,b$ is $\GL(n,\F)$ acting on $\Qq(a)$ by congruence.
\end{proof}
In particular Sylvester's theorem allows to count the number of $\Sp(V)$-orbits when the field $\F$ is real closed: since in this case the signature ${\rm sign}(Q)$ is a complete invariant of a quadratic form $Q$ up to equivalence (see \cite[Theorem 9]{Kap}), we have 
\begin{cor}
 Let $\F$ be a real closed field, and let $V$ be a symplectic $\F$-vector space of dimension $2n$. Then there are $n+1$ orbits of $\Sp(V)$ in $\Ll(V)^{(3)}$.
\end{cor}

A fundamental tool in the study of Lagrangian subspaces is the Kashiwara cocycle, which, at least  when $\F=\R$ is also known as the Maslov cocycle:
\begin{defn}
The Kashiwara cocycle is the function
$$\begin{array}{cccc}
\tau&\Ll^{3}(V)&\to&\Z\\
&(l_1,l_2,l_3)&\mapsto&{\rm sign}(Q)
\end{array}$$
where $Q$ is the quadratic form on the abstract direct sum $l_1\oplus l_2\oplus l_3$ defined by 
$$Q(x_1,x_2,x_3)=\langle x_1,x_2\rangle+\langle x_2,x_3\rangle+\langle x_3,x_1\rangle.$$
\end{defn}
The following properties of the Kashiwara cocycle are well known:
\begin{prop}[{See \cite[Section 1.5]{LV}}]\label{lem:1}
 Let $(V,\<\cdot,\cdot\>)$ be a $2n$-dimensional symplectic vector space over a real closed field:
 \begin{enumerate}
 \item $\tau$ is alternating and invariant for the diagonal action of $\Sp(V)$ on $\calL(V)^3$;
  \item $\tau$ has values in $\{-n,-n+1,\ldots,n\}$. On triples consisting of pairwise transverse Lagrangians it only achieves the values $\{-n,-n+2,\ldots,n\}$. If $|\tau(l_1,l_2,l_3)|=n$ then $l_i$ are pairwise transverse;
  \item If $(l_1,l_2,l_3)$ are pairwise transverse, then $$\tau(l_1, l_2, l_3) ={\rm sign}(Q_{l_1,l_2,l_3})$$
  \item $\tau$ is a cocycle: for each 4-tuple $(l_1,l_2,l_3,l_4)$, we have
  $$\tau(l_2,l_3,l_4)-\tau(l_1,l_3,l_4)+\tau(l_1,l_2,l_4)-\tau(l_1,l_2,l_3)=0.$$ 
 \end{enumerate}
\end{prop}

The second and the third statement in Proposition \ref{lem:1} justify the following definition:
\begin{defn}\label{def:2.11}
 A triple $(l_1,l_2,l_3)\in \Ll(V)^{(3)}$ is \emph{maximal} if $Q_{l_1,l_2,l_3}$ is positive definite. More generally an $n$-tuple $(l_1,\ldots,l_n)$ is \emph{maximal} if $Q_{l_i,l_j,l_k}$ is positive definite for any ordered triple of indices $i<j<k$.
\end{defn}

Let  $a,b\in\Ll(V)$ be transverse. We denote by $(\!(a,b)\!)$ the subset of $\Ll(V)$ consisting of points $c$ such that the triple $(a,c,b)$ is maximal:
$$(\!(a,b)\!)=\{c\in \Ll(V)|(a,c,b)\in \Ll(V)^{(3)} \text{ is maximal}\}.$$
Similarly for $a,b$ in $S^1$ we denote by $(\!(a,b)\!)$ the interval consisting of points $c$ such that $(a,c,b)$ is positively oriented.
The following lemma follows from Proposition \ref{lem:1} together with the observation that the unipotent radical of the stabilizer in $\Sp(V)$ of $b$ is isomorphic to $\Sym(n,\F)$ and  acts on $\Qq(a)$ by translation:
\begin{lem}\label{lem:2.10}
\begin{enumerate}
\item A triple $(a,l,b)$ is maximal  if and only if $l=j_{a,b}(q)$ for a positive definite quadratic form $q\in\Qq(a)$.
\item A 4-tuple $(l_1,l_2,l_3,l_4)$ is maximal if and only if $q_2$ and $q_3-q_2$ are positive definite where $l_2=j_{l_1,l_4}(q_2)$, $l_3=j_{l_1,l_4}(q_3)$.
\end{enumerate}
\end{lem}

We finish this subsection by analyzing the $\Sp(V)$-orbits in $\Ll(V)^{(4)}$. Using the objects and notations introduced in Section \ref{sec:models}, we fix a Lagrangian subspace $l_\infty$, a complex structure $J$ and a symplectic basis $\Bb$ of $V$ of the form $\Bb=\{e_1,\ldots,e_n,-Je_1,\ldots,-Je_n\}$. Moreover, when this does not seem to cause confusion, we suppress $\iota:{\rm Sym}(n,\F)\to \Ll(V)$  and simply represent an element in $\Ll(V)^{l_\infty}$ by an $\F$-valued symmetric matrix.
 
\begin{prop}\label{prop:4tran} Let $\F$ be a real closed field, and let $(l_1,l_2,l_3,l_4)\in\Ll(V)^{(4)}$ be a maximal 4-tuple. Then there exists a diagonal matrix $D=\diag(d_1,\ldots,d_n)$ satisfying $d_1\geq\ldots\geq d_n>0$, and an element $g_1\in\Sp(V)$ such that $$ g_1(l_1,l_2,l_3,l_4)=(-\Id,0,D,l_\infty).$$ 
Moreover there exists $g_2\in\Sp(V)$ such that $g_2(l_1,l_2,l_3,l_4)=(-\Id,\L,0,l_\infty)$ where $\L$ is diagonal with eigenvalues $-1<\l_i=-d_i/(1+d_i)<0$
\end{prop}
\begin{proof}
Since $\Sp(V)$ is transitive on maximal triples of  Lagrangians and the triple $(-\Id,0,l_\infty)$ is maximal, we can find an element $g_1\in\Sp(V)$ such that $g_1(l_1,l_2,l_3,l_4)=(-\Id,0,Z,l_\infty)$ for some positive definite matrix $Z$.

It is easy to verify that the stabilizer of the triple $(-\Id,0,l_\infty)$ in $\Sp(2n,\F)$ consists of matrices that have the form $\left\{\bsm A&0\\0&A\esm|\; A\in O(n)\right\}$ with respect to the basis $\Bb$ and acts by congruence. This allows to conclude: since $\F$ is real closed, every positive definite matrix is orthogonally congruent to a diagonal matrix $d=\diag(d_1,\ldots,d_n)$ with $d_1\geq\ldots\geq d_n$ (see \cite[Theorem 48]{Kap}). 

For the second assertion it is enough to take $$g_2=\bpm(\Id+D)^{-1/2}&-D(\Id+D)^{-1/2}\\0&(\Id+D)^{1/2}\epm g_1.$$
\end{proof}

An important role in the rest of the paper will be played by \emph{Shilov hyperbolic} elements of $\Sp(V)$.
We denote by $|\cdot|:\K\to \F^{>0}$ the absolute value $|a+ib|=\sqrt{a^2+b^2}$. 
\begin{defn}
 An element $g\in \Sp(V)$ is \emph{Shilov hyperbolic} if there exists a $g$-invariant decomposition $V=L_g^+\oplus L_g^-$, with  $L^\pm_g\in\Ll(V)$, such that all eigenvalues of the restriction of $g$ to $L_g^-$ have absolute value strictly smaller than one and all eigenvalues of the restriction of $g$ to $L_g^+$ have absolute value strictly bigger than one. In this case we denote by $M_g$ the restriction of $g$ to $L^+_g$.
\end{defn}

\begin{remark}\label{rem:Shyp}
When $V$ is a real vector space, the set of Lagrangians $\Ll(V)$ is the Shilov boundary of the symmetric space $\Tt_V$. Moreover if $g\in \Sp(V)$ is Shilov hyperbolic then there exists a Zariski open subset of $\Ll(V)$, the set of points transverse to $L_g^-$, which is contracted by $g$ to $L_g^+$. 
\end{remark}

\section{Representations admitting a maximal framing. The Collar Lemma}\label{sec:collar}
Let $\S$ be an oriented surface of genus $g$ and $p$ punctures. As mentioned in the introduction, we endow $\S$ with a complete hyperbolic metric of finite area and identify it with $\G\backslash\H^2$ where $\H^2$ is the Poincar\'e upper half plane. 

We now turn to the study of representations $\rho:\G\to \Sp(V)$  where $V$ is a symplectic space over a real closed field $\F$. We denote by $S \subseteq \deH^2$ any  $\G$-invariant subset containing all the fixed points of hyperbolic elements in $\G$.
\begin{defn}
We say that the representation $\rho:\G\to \Sp(V)$ admits a maximal framing if there exists an equivariant map $\phi:S\to \Ll(V)$ such that, whenever $x,y,z$ in $S$ are positively oriented, the triple of Lagrangians $(\phi(x),\phi(y),\phi(z))$ is maximal.
\end{defn}
\begin{remark}
It is a fundamental result \cite[Theorem 8]{BIW} that if $\F=\R$ then any maximal representation admits a maximal framing. In addition one can take $S=\deH^2$ and $\phi$ either left or right continuous.
\end{remark}

In this section we prove a generalization of the classical Collar Lemma of hyperbolic geometry to the context of representations which admit a maximal framing. In the case where $\F$ is the field of ordinary reals $\R$ this establishes a Collar Lemma for all maximal representations and gives a quantitative form of the fact due to Strubel \cite{Strubel} that for every hyperbolic element $\g$ in $\G$ the image $\rho(\g)$ is Shilov hyperbolic. 

\begin{thm}[Collar Lemma]\label{lem:collar}
If $\rho:\G\to\Sp(V)$ is a representation admitting a maximal framing, then for every hyperbolic element $\g$, $\rho(\g)$ is Shilov hyperbolic. Let $a,b$ be elements of $\G$ with positive intersection number and denote by $|\alpha_1|\geq\ldots\geq |\alpha_n|> 1$  the eigenvalues of the restriction of $\rho(a)$ to the attractive invariant Lagrangian $L^+_{\rho(a)}$ and analogously for $|\beta_1|\geq\ldots\geq |\beta_n|>1$ and $\rho(b)$. Then
\begin{enumerate}
 \item $(\det M_{\rho(a)}^{2/n}-1)(\det M_{\rho(b)}^{2/n}-1)\geq 1$;
 \item $ |\beta_1|^{2n}\geq{\displaystyle \frac 1{|\alpha_n|^{2}-1}}.$
\end{enumerate}
\end{thm}
We isolate a useful lemma which is used many times in the proof:
\begin{lem}\label{lem:ab}
 Let $M\in \GL_n(\F)$. Denote by $0<\tau_n\leq\ldots\leq \tau_1$ the eigenvalues of $M^t\!M$ and by $|\mu_n|\leq\ldots\leq |\mu_1|$ the absolute values of the eigenvalues of $M$. Then $\tau_n\leq |\mu_n|^2$.
 \end{lem}

 \begin{proof}
 If $S=M^t\!M$, then $S>\!>0$ and, if $(\cdot,\cdot ) $ denotes the standard scalar product, we have
  $$ \tau_n=\min_{v\neq 0}\frac{(Sv,v)}{(v,v)}$$
Since $(Sv,v)=(^t\!Mv,^t\!Mv)$ we get 
$\tau_n\leq\frac{(^t\!Mv,^t\!Mv)}{(v,v)}$ for every non zero $v$. If now $\mu_n$ belongs to $\F$, we get the statement applying this inequality to a corresponding eigenvector of $^t\!M$. If instead $\mu_n\in \K\setminus \F$, then there is a two dimensional subspace $E\cong \F^2$ in $\F^n$ which is invariant under $^t\!M$ and where this latter matrix acts like $\bsm a&b\\-b&a\esm$, for some $a,b\in \F$ with $a^2+b^2=|\mu_n|^2$. Then for $\bsm x\\y\esm\in E$ we have $$(^t\!M\bsm x\\y\esm,^t\!M\bsm x\\y\esm)=(ax+by)^2+(-bx+ay)^2=(a^2+b^2)(x^2+y^2)$$
which again implies the lemma.
 \end{proof}

\begin{proof}[Proof of Theorem \ref{lem:collar}]
Given two hyperbolic elements $a,b\in \G$, we denote by $\ax (a)$ and $\ax (b)$ the axes of $a$ and $b$, and by $a^+, b^+$ (resp. $a^-,b^-$) the attractive (resp. repulsive) fixed points of $a$ and $b$ in $\partial \H^2$.  

We can assume, without loss of generality, that $a$ and $b$ translate as represented by the picture and that the points $(a^-,b^-,ab^-,a^+,ab^+,ba^+,b^+,ba^-)$ are cyclically positively ordered (cfr. \cite[Lemma 2.2]{collar}).
\begin{center}
 \begin{tikzpicture}
  \draw (0,0) circle [radius=2];
\draw  (-2,0) node[left] {$a^-$} to [out=0,in=180]  (2,0) node [right]{$a^+$};
\node at (.8,0) {$>$};
\node [above]at (-1.1,0) {$\ax (a)$};
\draw (0,-2) node[below]{$b^-$} to [out=90,in=-90]  (0,2) node [above]{$b^+$};
\node at (0,.8) [rotate=90]{$>$};
\node [left] at (0,-1.3) {$\ax (b)$};
\draw (1.7,-1) node[below, right]{$a b^-$} to [out=150,in=-150]  (1.7,1) node [above,right]{$a b^+$};
 \draw (-1, 1.75) node[above left]{$b a^-$} to [out=-70,in=-110]  (1,1.75) node [above right]{$b a^+$};
 \end{tikzpicture}
\end{center}
Let $\phi:S\to \Ll(V)$ be the  maximal framing for $\rho$. Then the 6 points $$(\phi(b^-),\phi(a^+),\rho(a)\phi(b^+),\rho(b)\phi(a^+),\phi(b^+),\phi(a^-))$$ in $\Ll(V)^6$ form a maximal 6-tuple. This implies that they are pairwise transverse and every ordered subtriple forms a maximal triple. 

We are going to perform our computations in the upper half-space model. As in Section \ref{sec:models} fix a symplectic basis $\{e_1,\ldots,e_n,-Je_1,\ldots,-Je_n\}$ of $V$, set $l_\infty=\<e_1,\ldots,e_n\>$ and parametrize the set of Lagrangians transverse to $l_\infty$ by symmetric matrices. In view of Proposition \ref{prop:4tran}, we may, modulo conjugating $\rho$, assume that the 4-tuple $(\phi(a^-),\phi(b^-),\phi(a^+),\phi(b^+))$ is equal to $(-\Id,-\L^2,0,l_\infty)$ where $\L$ is diagonal with eigenvalues $0<\l_i<1$.
Since $\rho(a)$ fixes $0$ and $-\Id$ and $\rho(b)$ fixes $-\L^2$ and $l_\infty$ we have
$$\begin{array}{rl}
\rho(a)&=\bpm ^t\!A^{-1}&0\\-^t\!A^{-1}+A& A\epm\\
\rho(b)&=\bpm B&B\L^2-{\L^2}~^t\!B^{-1}\\0& ^t\!B^{-1}\epm
\end{array}$$
for some matrices $A,B$. Let $\{\alpha_1,\ldots,\alpha_n\}$ and $\{\beta_1,\ldots,\beta_n\}$ denote the eigenvalues of $A$ (resp. $B$) counted with multiplicity and ordered so that $|\alpha_i|\geq |\alpha_{i+1}|$ and similarly $|\beta_i|\geq|\beta_{i+1}|$.

An easy computation gives 
$$\begin{array}{rcl}\rho(b)\phi(a^+)&=\rho(b)\cdot 0&=-\L^2+B{\L^2}~^t\!B\\
\rho(a)\phi(b^+)&=\rho(a)\cdot l_\infty&=(A^t\!A-\Id)^{-1}.
\end{array}$$
We summarize this information in the following picture for the reader's convenience:
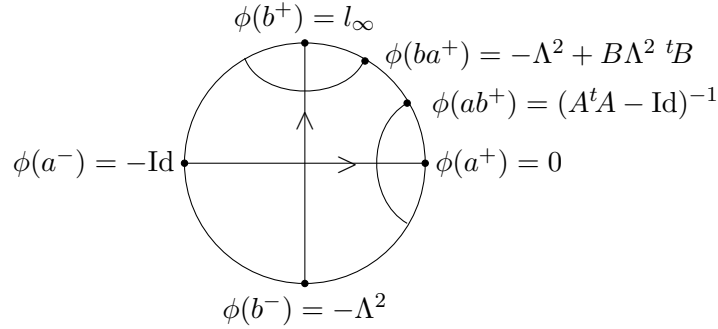
\begin{center}
 \begin{tikzpicture}[scale=.8]
  \draw (0,0) circle [radius=2];
\filldraw (-2,0) circle [radius=1.5pt];
\filldraw (0,2) circle [radius=1.5pt];
\filldraw (2,0) circle [radius=1.5pt];
\filldraw (0,-2) circle [radius=1.5pt];
\filldraw (1.7,1) circle [radius=1.5pt];
\filldraw (1,1.7) circle [radius=1.5pt];
\draw  (-2,0) node[left] {$\phi(a^-)=-\Id$} to [out=0,in=180]  (2,0) node [right]{$\phi(a^+)=0$};
\node at (.7,-0.02) {$>$};
\node [above]at (-.8,0) {};
\draw (0,-2) node[below]{$\phi(b^-)=-\L^2$} to [out=90,in=-90]  (0,2) node [above]{$\phi(b^+)=l_\infty$};
\node at (0,.7) [rotate=90]{$>$};
\node [left] at (0,-.8) {};
\draw (1.7,-1) node[below right]{} to [out=150,in=-150]  (1.7,1) ;
 \draw (-1, 1.75) node[above left]{} to [out=-70,in=-110]  (1,1.75) ;
\node at (4.5,1) {$\phi(a b^+)=(A^t\!A-\Id)^{-1}$};
\node at  (3.9,1.8) {$\phi(b a^+)=-\L^2+B{\L^2}~^t\!B$};
 \end{tikzpicture}
\end{center}

The maximality of the triple 
$$(\phi(a^+),\phi(ab^+),\phi(b^+))=(0,(A ^t\!A-\Id)^{-1},l_\infty)$$
 implies that the quadratic form represented by $(A ^t\!A-\Id)^{-1}$ is positive definite and in particular all the eigenvalues of $A ^t\!A$ are bigger than one.
Thus if we denote by $\tau_1\geq\ldots\geq \tau_n>1$ the eigenvalues of $A ^t\!A$, it follows from Lemma \ref{lem:ab} that $1<\tau_n\leq|\alpha_n|^{2}$ and hence we get that the eigenvalues of $A$ satisfy $1<|\alpha_n|\leq\ldots\leq |\alpha_1|$; in particular $\rho(a)$ is Shilov hyperbolic.

We now exploit the maximality of the triple 
$$(\phi(a^+),\phi(ba^+),\phi(b^+))=(0,{B\L^2}\!~^t\!B-\L^2,l_\infty),$$
which is equivalent to the fact that the quadratic form
$$\L((\L^{-1}B\L) ^t\!(\L^{-1}B\L)-\Id)\L=B\L^2\!~ {^t\!B}-\L^2$$
is positive definite. Denoting by $C$ the matrix $\L^{-1}B\L$ we get that all the eigenvalues of $C^t\!C$ are bigger than 1. 
Let $1<\sigma_n\leq\ldots\leq \sigma_1$ denote the eigenvalues of $C ^t\!C$. From Lemma \ref{lem:ab} we get that the  eigenvalues of $B$ satisfy as $1<|\beta_n|\leq\ldots\leq| \beta_1|$.  This implies that $\rho(b)$ is Shilov hyperbolic as well. Moreover we have 
$$\sigma_1\leq \det(C ^t\!C)=\det(C)^2\leq |\beta_1|^{2n}.$$

Last we exploit the maximality of the quadruple
$$(\phi(a^+),\phi(ab^+),\phi(ba^+),\phi(b^+))=(0,(A^t\!A-\Id)^{-1},{B\L^2}~^t\!B-\L^2,l_\infty)$$
which is equivalent to the property that
\begin{equation}\label{eqn:5}\L((\L^{-1}B\L) ^t\!(\L^{-1}B\L)-\Id)\L-(A ^t\!A-\Id)^{-1}>\!>0.
\end{equation}
(see Lemma \ref{lem:2.10} (2)).

Taking into account that $1< \sigma_n\leq\ldots\leq \sigma_1$, we obtain that
if $x_n\leq\ldots\leq x_1$ are the  eigenvalues of 
$$X=\L((\L^{-1}B\L) ^t\!(\L^{-1}B\L)-\Id)\L=\L(C^t\!C-\Id)\L,$$ 
then (\ref{eqn:5}) implies
\begin{equation}\label{eqn:1}
x_i\geq \frac 1{\tau_{n+1-i}-1}, \quad \text{ for all $1\leq i\leq n$.}
\end{equation}

Next we claim that $x_i<(\s_i-1)$. Indeed, by the minmax theorem, we have 
\begin{eqnarray}
x_k&=&\min_{\dim W=n+1-k}\max_{v\in W}\frac{\langle \L(C ^t\!C-\Id)\L v,v\rangle}{\|v\|^2}=\nonumber\\
&=&\min_{\dim W=n+1-k}\max_{v\in W}\left(\frac{\langle (C ^t\!C-\Id)\L v,\L v\rangle}{\|\L v\|^2} \frac{\|\L v\|^2}{\|v\|^2}\right)\leq\nonumber\\
&\leq&(\sigma_k-1)\max_{v\in\F^n} \frac{\|\L v\|^2}{\|v\|^2}=(\sigma_k-1)\l_n^2<\s_k-1\nonumber
  \end{eqnarray}
where the last inequality takes into account that $\l_n<1$.

Setting $i=1$ in the above inequalities we obtain $\s_1-1\geq \frac1{ \tau_n-1}$ which, together with the inequalities previously obtained, namely that $|\beta_1|^{2n}\geq \s_1$ and $\tau_n\leq |\alpha_n|^2$, shows assertion (2).

We establish now the inequality (1). Since $x_i<\s_i-1$, we get
  $$
     (\det B)^2=\prod_{i=1}^n\s_i>\prod_{i=1}^n (1+x_i)$$
and we deduce from (\ref{eqn:1}) that 
$$\prod_{i=1}^n (1+x_i)
\geq\prod_{i=1}^n\frac{\tau_i}{ (\tau_i-1)}.
$$

Since over any real closed field $\F$, and for any $a_1,\ldots,a_n>1$ one has
$\prod_{i=1}^n (a_i^n-1)\leq \left(a_1a_2\ldots a_n-1\right )^n$ (cfr. Appendix 1), we deduce, choosing $a_i=\tau_i^{1/n}$,
$$\left(\prod_{i=1}^n\frac{\tau_i} {(\tau_i-1)}\right)^{1/n}\geq\frac {(\tau_1\ldots\tau_n)^{1/n}}{(\tau_1\ldots\tau_n)^{1/n}-1}.$$
Using $\tau_1\ldots\tau_n=(\det A)^2$, this establishes the first inequality.

\end{proof}
\begin{remark}\label{rem:collar}
 In the specific case of a maximal representation with values in $\Sp(2n,\R)$ and which in addition belongs to the Hitchin component, assertion (2) is a weaker version of the Collar Lemma for Hitchin representations proven by Lee and Zhang \cite{collar}: their result implies, under these hypotheses, that 
 $$\beta_1^2\geq\frac{\alpha_n^{2}}{(\alpha_n^2-1)}.$$
 This is Proposition 2.12 (1) in their paper.
\end{remark}

\section{Crossratios and the geometry of $\F$-tubes}\label{sec:3}
\subsection{Crossratios}\label{sec:cr}
We now introduce a useful tool to study the geometry of the Siegel space. 
Let $V$ be a $2n$ dimensional vector space, over a field $\L$. Observe that if $a,b$ are $n$-dimensional subspaces which are transverse ($a\tra b$), then we have a direct sum decomposition $V=a\oplus b$ and thus we can define the projection $p_a^{/\!/b}:V\to a$ onto $a$ parallel to $b$. Let now $(l_1,l_2,l_3,l_4)$ be a quadruple in $\Gr_n(V)$ with the property that $l_1\tra l_2$, $l_3\tra l_4$.
\begin{defn}\label{defn:cr}
 The crossratio of $(l_1,l_2,l_3,l_4)$ is the endomorphism of $l_1$ defined by
$$R(l_1,l_2,l_3,l_4)=p_{l_1}^{/\!/l_2}\circ p_{l_4}^{/\!/l_3}|_{l_1}.$$
\end{defn}
The crossratio has the following equivariance property:  for all $g\in Gl(V)$, we have $$R(gl_1,gl_2,gl_3,gl_4)=gR(l_1,l_2,l_3,l_4)g^{-1}.$$

It will be useful, in the following, to have an explicit expression for $R$ once a basis $\calB=\{e_1,\ldots,e_{2n}\}$ of $V$ is fixed. Recall that, as in Section \ref{sec:models}, the choice of the basis $\calB$ allows us to represent an element $m$ of $\Gr_n(V)$ with a $2n\times n$ matrix $M$ of maximal rank: the columns of the matrix $M$ are understood to be the coordinates, with respect to $\calB$, of a basis of $m$. With these notations we have the following
\begin{lem}\label{lem:a1}
 Let us assume that the columns of the matrix $\bsm X_i\\\Id_n\esm$ form a basis $\calB_i$ of the $n$-dimensional vector space $l_i$. Then the expression for $R(l_1,l_2,l_3,l_4)$ with respect to the basis $\Bb_1$ of $l_1$ is given by
 $$R(l_1,l_2,l_3,l_4)=(X_1-X_2)^{-1}(X_4-X_2)(X_4-X_3)^{-1}(X_1-X_3).$$
\end{lem}
\begin{proof}
 The matrix representing the linear map $p_{l_4}^{/\!/l_3}|_{l_1}$ with respect to the bases $\Bb_1$ of $l_1$ and $\Bb_4$ of $l_4$ is the unique $A\in Gl_n(\L)$ such that
 $$\bsm X_1\\\Id_n\esm=\bsm X_4\\\Id_n\esm A+\bsm X_3\\\Id_n\esm (\Id-A).$$
 Solving for $A$ we obtain,
 $$A=(X_4-X_3)^{-1}(X_1-X_3).$$
 Notice that $X_4-X_3$ is invertible since by assumption $l_3$ and $l_4$ are transverse.
 
 Similarly we get that the matrix representing the restriction of the linear map $p_{l_1}^{/\!/l_2}$ to ${l_4}$ with respect to the bases $\Bb_4$ of $l_4$ and $\Bb_1$ of $l_1$ is given by
  $$B=(X_1-X_2)^{-1}(X_4-X_2).$$
  Since, by definition, the endomorphism $R(l_1,l_2,l_3,l_4)$ is the composition of $p_{l_1}^{/\!/l_2}$ and  $p_{l_4}^{/\!/l_3}|_{l_1}$, and $p_{l_4}^{/\!/l_3}|_{l_1}$ has image contained in $l_4$, we get that 
  $$R(l_1,l_2,l_3,l_4)=BA$$
  which gives the desired result.
\end{proof}

Let us now fix a basis $\Bb$ of $V$, set, as usual, $l_\infty=\<e_1,\ldots,e_n\>$ and represent with a matrix $M\in M_n(\L)$ the subspace spanned by the columns of $\bsm M\\\Id\esm$. By a similar computation we have

\begin{lem}\label{lem:cr}
 Assume $0$, $Z$, $X$, $l_\infty$ 
are pairwise transverse,
 $$R\left(0,Z,X,l_\infty\right)=Z^{-1}X.$$
\end{lem}
It will be useful to understand how the crossratio varies with respect to permutations of the factors. In particular we need to be able to compare endomorphisms of different vector spaces. Given two vector spaces $l_1,l_2$ of the same dimension we say that two endomorphism $R_1\in\End(l_1)$ and $R_2\in \End(l_2)$ are \emph{conjugate} if there exists an isomorphism $g:l_1\to l_2$ such that $gR_1g^{-1}=R_2$. In this case we write $R_1\cong R_2$.
\begin{lem}\label{lem:2.7}
Assume that the subspaces $l_i$ are pairwise transverse, then 
\begin{enumerate}
\item $R(l_1,l_2,l_4,l_3)= \Id-R(l_1,l_2,l_3,l_4). $
\item $R(l_4,l_1,l_2,l_3)\cong(\Id-R(l_1,l_2,l_3,l_4)^{-1})^{-1};$
\item $R(l_1,l_4,l_2,l_3)\cong R(l_2,l_3,l_1,l_4)\cong(\Id-R(l_1,l_2,l_3,l_4))^{-1}.$
\end{enumerate}

\end{lem}
\begin{proof}
(1) By definition we have 
$$
\begin{array}{l}
p_{l_1}^{/\!/l_2}\circ p_{l_4}^{/\!/l_3}|_{l_1}+p_{l_1}^{/\!/l_2}\circ p_{l_3}^{/\!/l_4}|_{l_1}=\\
=p_{l_1}^{/\!/l_2}\circ (p_{l_4}^{/\!/l_3}+p_{l_3}^{/\!/l_4})|_{l_1}=\\
=p_{l_1}^{/\!/l_2}\circ\Id|_{l_1}=\Id_{l_1}.
\end{array}
$$

(2) Up to the $Gl(V)$ action we can assume that $l_1=0$, $l_2= Z$, $l_3= X$ and $l_4=l_\infty$.  In particular  $R(0,Z,X,l_\infty)=Z^{-1}X$. 
In order to compute $R(Z,X,l_\infty,0)$ we  compute
$p_{0}^{/\!/l_\infty}|_{Z}=\Id$ and  $p_{Z}^{/\!/X}|_{0}=(\Id+X^{-1}Z)^{-1}$.

(3) Similarly one gets that $p_{X}^{/\!/Z}|_{0}=(\Id-Z^{-1}X)^{-1}$. The second equivalence follows from the fact that $p_{l_\infty}^{/\!/0}|_{Z}=Z$ and $p_{Z}^{/\!/X}|_{l_\infty}=(Z-X)^{-1}$.
\end{proof}

\subsection{$\F$-Tubes}
Let $(V,\<\cdot,\cdot\>)$ be a symplectic vector space over a real closed field $\F$. Recall from Section \ref{sec:models} that $\K$ denotes the quadratic extension $\F[i]$, that $\s:\Ll(V_\K)\to\Ll(V_\K)$ is induced by the complex conjugation with respect to the real structure $V$ of $V_\K$ and that $\Tt_V$ is the model of the Siegel space contained in $\Ll(V_\K)$.
For any pair of transverse Lagrangians $(a,b)$ in $\Ll(V)^{(2)}$, we introduce here an algebraic subset  $\Yy_{a,b}$  of the Siegel space $\Tt_V$ that is determined by the pair $(a,b)$ and whose dimension is  half the dimension of $\Tt_V$. We call such subsets $\F$-tubes. In the case when $\F=\R$, the subsets $\Yy_{a,b}$ are Lagrangian submanifolds of the same rank as $\Xx_\R$; the $\F$-tube $\Yy_{a,b}$ can be seen as the higher rank generalization of a geodesic of the Poincar\'e model which is more suited to our purposes.

With the notation of Section \ref{sec:basic} we define
$$\Yy_{a,b}=\{l\in\Tt_V|\; R(a,l, \s(l),b)=-\Id\}.$$
Notice that requiring that an endomorphism of a vector space is equal to $-\Id$ does not depend on the choice of a basis.
From the equivariance property of the crossratio and the fact that the symplectic group commutes with the complex conjugation $\s$ we deduce that

 \begin{equation}\label{eqn:3} g\Yy_{a,b}=\Yy_{ga,gb}, \; \text{ for any }g\in \Sp(V).\end{equation}

Our first goal is to give equations for $\Yy_{a,b}$ in the Siegel upper half-space for some specific choice of the pair $(a,b)$. 
\begin{lem}\label{lem:2.6}
The $\F$-tube with endpoints $0,l_\infty$ is 
$$\Yy_{0,l_\infty}=\{iY|\; Y\in {\rm Sym}^+(n,\F)\}.$$
\end{lem}
\begin{proof}
It follows from Lemma \ref{lem:cr} that 
 $R(0,Z, \s(Z),l_\infty)=Z^{-1}\ov Z.$ 
Clearly we have $Z^{-1}\ov Z=-\Id$ if and only if $\ov Z=-Z$ and this concludes the proof.
\end{proof}
An immediate consequence of Lemma \ref{lem:2.6} and the equivariance property (\ref{eqn:3}) is that if $\F$ is a real closed field, the stabilizer of $\Yy_{a,b}$ is isomorphic to $\GL(n,\F)$ and it acts transitively on $\Yy_{a,b}$.

It will also be useful to have explicit expression for the set $\Yy_{a,b}$ when $a$ and $b$ are transverse to $l_\infty$. This has a particularly nice expression when $a=\<e_1-e_{n+1},\ldots,e_n-e_{2n}\>$ and $b=\<e_1+e_n,\ldots,e_n+e_{2n}\>$:
\begin{lem}\label{lem:2.20}
 If $a,b\in \Ll(V)$ correspond to the matrices $-\Id$ and $\Id$, then 
 $$\begin{array}{rl}
\Yy_{-\Id,\Id}&=U(n)\cap \Xx_\F\\
&=\{X+iY\in\Xx_\F|\;YX=XY,\; X^2+Y^2=\Id\}.
\end{array}$$

\end{lem}
\begin{proof}
Lemma \ref{lem:a1} implies:
 $$R(-\Id,Z,\sigma(Z),\Id)=(-\Id-Z)^{-1}(\Id- Z)(\Id-\ov Z)^{-1}(-\Id-\ov Z).$$
Since $\Id+\ov Z$ and $(\Id- \ov Z)^{-1}$ commute, the equality $R(-\Id,Z,\sigma(Z),\Id)=-\Id$ reads 
$$(\Id- Z)(\Id+\ov Z)=-(\Id+Z)(\Id-\ov Z)$$
which implies 
$$\Id- Z+\ov Z-Z\ov Z=-\Id+\ov Z-Z+Z\ov Z,$$
and hence, $Z\ov Z=Z^*\!Z=\Id$.
\end{proof}
 As a consequence of the explicit parametrization of the sets $\Yy_{0,l_\infty}$ and $\Yy_{-\Id,\Id}$ we obtain

\begin{prop}
Assume that $\F$ is a real closed field. Let $(a,b,c,d)\in \Ll(V)^{(4)}$ be a maximal 4-tuple. The $\F$-tubes $\Yy_{a,c}$ and $\Yy_{b,d}$ meet exactly in one point.
\end{prop}

\begin{proof}
 Up to the symplectic group action we can assume that $(a,b,c,d)=({-\Id},0,{D},l_\infty)$ for some diagonal matrix $D=\diag(d_1,\ldots d_n)$ with $d_i>0$  (see Proposition \ref{prop:4tran}). Let $y$ be a point in $\Yy_{0,l_\infty}\cap\Yy_{-\Id,D,d}$. Since $y$ belongs to $\Yy_{0,l_\infty}$ we know that $y$ has expression $y=iY$ for some positive definite matrix $Y$. From the definition of $\Yy_{-\Id,D}$ we get 
$$(-\Id-iY)^{-1}(D- iY)(D+iY)^{-1}(-\Id+iY)=-\Id.$$
This is equivalent to
$$(D- iY)(D+iY)^{-1}=(\Id+iY)(-\Id+iY)^{-1}$$
which in turn, using that $(\Id+iY)$ and $(-\Id+iY)^{-1}$ commute restates as
$$(-\Id+iY)(D- iY)=(\Id+iY)(D+iY).$$
This last equation reads $Y^2=D$ which has a unique positive solution.

\end{proof}
\begin{remark}
 If the ordered field $\F$ is not real closed, one can similarly get that, if $(a,b,c,d)$ is maximal, the $\F$-tubes $\Yy_{a,c}$ and $\Yy_{b,d}$ meet in at most one point and they intersect if all the eigenvalues of the crossratio $R(a,b,c,d)$ are squares in $\F$.
\end{remark}

\subsection{Reflection with respect to $\Yy_{a,b}$}\label{sec:ort}
In this subsection we introduce a notion of orthogonality for $\F$-tubes and establish that the set of $\F$-tubes orthogonal to a given one foliate the space $\calT_V$. Our main tool will be the characterization of $\Yy_{a,b}$ as the fixed point set of an involution $\s_{a,b}$ which we now define.
Let $a,b$ be transverse Lagrangians in $\Ll(V)$, we consider the real form $V_{a,b}$ of $V_\K$ given by 
$$V_{a,b}=\<v+iw|v\in a, w\in b\>,$$
and denote by $\s_{a,b}$ the complex conjugation of $V_\K$ fixing $V_{a,b}$.
The following properties of $\s_{a,b}$ can be checked easily:
\begin{lem}\label{lem:2.12}
\begin{enumerate}
 \item $\s_{a,b}$ is $\K$-antilinear;
 \item $\s_{a,b}\circ \s=\s\circ\s_{a,b}$,  in particular $\s_{a,b}$ preserves $V$;
 \item $\<\s_{a,b}(\cdot),\s_{a,b}(\cdot)\>_\K=-\ov{\<\cdot,\cdot\>}_\K$;
 \item for every $g$ in $\Sp(V)$ we have $g\s_{a,b}=\s_{ga,gb}g$.
\end{enumerate}
\end{lem}
As a consequence of the first fact of Lemma \ref{lem:2.12} we get that $\s_{a,b}$ induces a map on $\Gr_n(V)$ that, with a slight abuse of notation, will be also denoted by $\s_{a,b}$. The third fact of Lemma \ref{lem:2.12} implies that $\s_{a,b}$ restricts to a map 
$$\s_{a,b}:\Ll(V_\K)\to \Ll(V_\K),$$
which preserves the subspaces we are interested in:
\begin{lem}
 The involution $\s_{a,b}$ preserves the subspaces $\Tt_V$ and $\Ll(V)$ of $\Ll(V_\K)$. It commutes with the crossratio.
\end{lem}
\begin{proof}
 Since the $\F$-linear map $\s_{a,b}$ preserves $V$, the induced map on $\Ll(V_\K)$ preserves the subspace $\Ll(V)$. The fact that $\s_{a,b}$ induces a map of $\Tt_V$ follows from the following computation which uses Lemma \ref{lem:2.12} (3): for every $v,w\in V_\K$
 $$\begin{array}{rl}
 i\<\s_{a,b}(v),\s_{a,b}(w)\>_\K&=-i\ov{\<v,w\>}_\K\\
 &=\ov{i\<v,w\>}_\K.
 \end{array}$$
 In particular the restriction of $i\langle\cdot,\sigma(\cdot)\rangle_\K$ to a Lagrangian $l\in \Ll(V_\K)$ is positive definite if and only if its restriction to $\s_{a,b}(l)$ is.
 
For any pair $a,b\in\Ll(V)^{(2)}$ and for any 4-tuple $(l_1,l_2,l_3,l_4)$ in the domain of definition of $R$  we have
  $$\s_{a,b}R(l_1,l_2,l_3,l_4)\s_{a,b}=R(\s_{a,b}(l_1),\s_{a,b}(l_2),\s_{a,b}(l_3),\s_{a,b}(l_4)):$$
  this follows from the equivariance property of the crossratio and the fact that $\s_{a,b}^2=\Id$.
\end{proof}

It is easy to check from the very definition of $\s_{0,l_\infty}$ that for any $Z\in \Xx_\F$  we have 
 $\s_{0,l_\infty}(Z)=-\ov Z$. In particular $\Yy_{0,l_\infty}=\Tt_V\cap {\rm Fix}(\s_{0,l_\infty})$. An immediate corollary of the transitivity of the symplectic group action on $\Ll(V)^{(2)}$ is the following:
\begin{cor}
 For any pair $(a,b)$ we have $\Yy_{a,b}=\Tt_V\cap {\rm Fix}(\s_{a,b})$. 
\end{cor}
Another useful characterization of the $\F$-tubes is the following:
\begin{lem}\label{lem:Y}
In the model $\mathbb X_V$, 
$$\Yy_{a,b}=\{J\in\mathbb X_V: a \text{ and }b \text{ are orthogonal for }\<J\cdot,\cdot\>\}.$$
\end{lem}
\begin{proof}
In the notation of Section \ref{sec:1}, let $J\in\mathbb X_V$, then $\s_{a,b}(L_J^+)=L_J^+$ if and only if $\s_{a,b}(L_J^-)=L_J^-$ and hence, since $\s_{a,b}$ is $\K$-antilinear, we deduce $\s_{a,b}(J\otimes \mathbb I _\K)=-(J\otimes \mathbb I _\K)\s_{a,b}$, which, by restriction to $V=a\oplus b$, is equivalent to $\s_{a,b}J=-J\s_{a,b}$; the latter is equivalent to $J(a)=b$ that is $a$ and $b$ are orthogonal with respect to $\<J\cdot,\cdot\>$.
\end{proof}

The restriction of $\s_{a,b}$ to the subset of $\Ll(V)$ consisting of points that are transverse to $a$ and $b$ can also be characterized in term of the crossratio:
\begin{prop}\label{prop:2.23}
For each $c\in \Ll(V)$ transverse to $a$ and $b$, $\s_{a,b}(c)$ is the unique point satisfying
$$R(a,c,\s_{a,b}(c),b)=-\Id.$$
\end{prop}
\begin{proof}
Up to the symplectic group action  we can assume that $a=0$ and $b=l_\infty$. Since $c$ is transverse to $l_\infty$, it can be represented by a symmetric matrix $S$ with coefficients in $\F$. The formula of Lemma \ref{lem:cr} implies that $R(0,S,\s_{0,l_\infty}(S),l_\infty)=S^{-1}\s_{0,l_\infty}(S)$ and hence the unique point satisfying $R(0,S,\s_{0,l_\infty}(S),l_\infty)=-\Id$ is $-S$.
\end{proof}
When $\F=\R$, and the 4-tuple $(a,b,c,d)$ is maximal, the two $\R$-tubes $\Yy_{a,c}$ and $\Yy_{b,d}$ are orthogonal as totally geodesic submanifolds of the Riemannian manifold $\Xx_\R$ precisely when $R(a,b,c,d)=2\Id$. For arbitrary real closed fields we take this property as definition of orthogonality.
\begin{defn}\label{defn:4.12}
Let $(a,b,c,d)$ be maximal. Two $\F$-tubes $\Yy_{a,c}$ and $\Yy_{b,d}$ are \emph{orthogonal} if  $R(a,b,c,d)=2\Id$. In this case we write $\Yy_{a,c}\bot\Yy_{b,d}$.
\end{defn}

Notice that the orthogonality relation is symmetric since $R(d,a,b,c)$ is conjugate to $(\Id-R(a,b,c,d)^{-1})^{-1}$ (cfr. Lemma \ref{lem:2.7} (2)). The following lemma is a consequence of the property of the crossratio established in Lemma \ref{lem:2.7} (1) and the characterization of the involution $\s_{a,b}$ in term of the crossratio given in Proposition \ref{prop:2.23}:
\begin{lem}\label{lem:4.15}
 Let $(a,b,c,d)$ be a maximal quadruple. The following are equivalent:
 \begin{enumerate}
  \item $\Yy_{a,c}\bot\Yy_{b,d}$;
  \item $d=\s_{a,c}(b)$;
  \item $c=\s_{b,d}(a)$.
 \end{enumerate}
\end{lem}
We now turn to an important geometric feature of the Siegel upper half-space, namely that the $\F$-tubes orthogonal to any fixed $\F$ tube foliate the whole space. We first verify this in a special case:
\begin{prop}\label{prop:ort}
Assume that $\F$ is real closed. For any $Z=X+iY\in \Tt_V$ there exists a unique $S$ in $\Ll(V)^{l_\infty}$ such that $(0,S,l_\infty)$ is maximal and $Z\in \Yy_{-S,S}$. Moreover $S$ is given by
$$S=Y^{1/2}\sqrt{\Id+(Y^{-1/2}XY^{-1/2})^2}Y^{1/2}.$$
\end{prop}
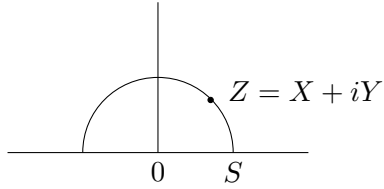
\begin{center}
\begin{tikzpicture}
\draw (-2,0 ) to (2,0);
\draw (0,0) to(0,2);
\draw (-1,0) to [out=90, in=180] (0,1) to [out=0, in=90] (1,0 );
\filldraw  (.7,.7) circle [radius=1pt];
\node [above, right]at (.8,.8) {$Z=X+iY$};
\node [below] at (0,0) {$0$};
\node [below] at (1,0) {$S$};
\end{tikzpicture}
\end{center}
\begin{proof}
Given $Z=X+iY$ we look for a positive definite matrix $S$ with $Z\in \Yy_{-S,S}$. Denoting by $a(S^{-1/2})$ the element of $\Sp(2n,\F)$ represented by the matrix $\bsm S^{-1/2}&0\\0&S^{1/2}\esm$, we have $a(S^{-1/2})\Yy_{-S,S}=\Yy_{-\Id,\Id}$. The condition $a(S^{-1/2})Z\in\Yy_{-\Id,\Id}$ leads, in view of the equations of Lemma \ref{lem:2.20}, to:
$$\left\{\begin{array}{l}
(S^{-1/2}XS^{-1/2})(S^{-1/2}YS^{-1/2})=(S^{-1/2}YS^{-1/2})(S^{-1/2}XS^{-1/2})\\
(S^{-1/2}XS^{-1/2})^2+(S^{-1/2}YS^{-1/2})^2=\Id
\end{array}\right.$$
From the first equation we get, observing that $Y$ is invertible,  $$XS^{-1}=YS^{-1}XY^{-1}.$$ Substituting this last equality in the second equation, and defining the matrix $V:=Y^{-1/2}SY^{-1/2}$, we get
$$V^{-1}((Y^{-1/2}XY^{-1/2})^2+\Id)=V$$
which implies 
$$V=\sqrt{\Id+(Y^{-1/2}XY^{-1/2})^2}$$
and 
$$S=Y^{1/2}\sqrt{\Id+(Y^{-1/2}XY^{-1/2})^2}Y^{1/2}.$$
This shows the formula and implies uniqueness.

\end{proof}
Since all $\F$-tubes are $\Sp(V)$-conjugate, we obtain:
\begin{cor}\label{cor:ort}
 For any transverse pair $(a,b)\in\Ll(V)^{(2)}$ and any $z\in\Tt_V$, there exists an unique $c\in\Ll(V)$ such that $(a,c,b)$ is maximal and $z$ belongs to $\Yy_{c,\s_{a,b}(c)}$.
\end{cor}

Corollary \ref{cor:ort} allows us to define the orthogonal projection

$$
\pr_{\Yy_{a,b}}:\Tt_V\cup (\!(a,b)\!) \cup (\!(b,a)\!)\to \Yy_{a,b}
$$
as follows
\begin{enumerate}
\item  if $c\in (\!(a,b)\!)\cup(\!(b,a)\!)$,  then we set $\pr_{\Yy_{a,b}}(c)=\Yy_{c,\,\s_{a,b}(c)}\cap \Yy_{a,b}$
\item if $Z\in\Tt_V$, then we set $\pr_{\Yy_{a,b}}(Z)=\Yy_{c,\s_{a,b}(c)}\cap \Yy_{a,b}$, where $c$ is the unique real Lagrangian such that $Z\in \Yy_{c,\s_{a,b}(c)}$.
\end{enumerate}

It is easy to check that, when restricted to its set of definition in $\Ll(V)$, the orthogonal projection respects crossratios:
\begin{lem}
 Let $(a,b)$ be a pair of transverse Lagrangians, and let $x,y$ be points in $(\!(a,b)\!)$. Then we have 
 $$R(a,x,y,b)=R(a,\pr_{\Yy_{a,b}}(x),\pr_{\Yy_{a,b}}(y),b).$$
\end{lem}
\begin{proof}
 Up to the symplectic group action we can assume that $a=0, b=l_\infty$. In that case the result follows from the explicit formula for the crossratio and for the orthogonal projection.
\end{proof}

\section{Reduction modulo an order convex subring}\label{sec:4}
\subsection{Order convex subrings}
Let $\F$ be a real closed, non Archimedean field. We denote by $\Oo<\F$ an \emph{order convex} subring. This means that $\Oo$ is a subring with the additional property that for every positive element $x$ in $\F$, if there exists $y$ in $\Oo$ with $0<x<y$, then $x$ belongs to $\Oo$ as well. It is easy to verify that in this case $\Oo$ is a local ring whose maximal ideal $\Ii$ is given by 
$$\Ii=\{x\in \Oo|\; x^{-1}\notin \Oo\}.$$ 
We will denote by $\F_{_\Oo}$ the quotient field $\F_{_\Oo}:=\Oo/\Ii$. The field $\F_{_\Oo}$ is real closed as well.  
The following examples of order convex subrings will play an important role in the sequel:
\begin{example}\label{ex:4.17}
Let $\s \in \F$ be an infinitesimal: this means that $\s$ is a positive element satisfying $\s<1/n$ for any integer $n$. An example of an order convex subring of $\F$ is given by the set of elements comparable to $\s$: 
$$\Oo_\s=\{x\in \F||x|<\s^{-k} \text{ for some $k\in \N$}\}$$
in this case the maximal ideal can also be characterized as
$$\Ii_\s=\{x\in \F||x|<\s^{k} \text{ for all $k\in \N$}\}.$$
\end{example}
\begin{example}
 Let us assume that $\F$ admits an order compatible valuation $v$. An example of order convex subring is given by the elements with positive valuation
 $$\Uu=\{x\in \F|\; v(x)\geq 0\}$$
 and the maximal ideal can be characterized as
 $$\Mm=\{x\in\F|\;v(x)>0\}.$$
\end{example}
\subsection{$\Oo$-points}
Let $\Oo$ be an order convex subring of $\F$ and let $W$ be a finite dimensional $\F$-vector space equipped with an $\F$-valued scalar product $(\cdot,\cdot)$. Then we set
$$W(\Oo )=\{v\in W|\;(v,v)\in \Oo \}$$
and
$$W(\Ii)=\{v\in W|\;(v,v)\in \Ii\}$$ are $\Oo$-submodules; if $e_1,\ldots,e_m$ is any orthonormal basis of $W$, then one verifies that
$$\begin{array}{ccc}W(\Oo )=\sum_{i=1}^m\Oo e_i&\text{
and} &W(\Ii)=\sum_{i=1}^m\Ii e_i.\end{array}$$
This implies that the quotient $W_{_\Oo}=W(\Oo )/W(\Ii)$ is a $\F_{_\Oo}$-vector space of dimension $m=\dim(W)$, that the scalar product $(\cdot,\cdot)$ descends to a well defined scalar product $(\cdot,\cdot)_{_\Oo}$ on $W_{_\Oo}$ and that, if $p_{_\Oo} :W(\Oo)\to W_{\Oo}$ denotes the quotient map, $\{p_{_\Oo}(e_1),\ldots,p_{_\Oo}(e_m)\}$ is again an orthonormal basis of $W_{_\Oo}$.
Notice, however, that the map $p_{_\Oo}$ depends on the choice of the scalar product on $W$.

The subgroup $$\GL(W)(\Oo):=\{g\in\GL(W)|\;g(W(\Oo))=W(\Oo)\}$$
preserves $W(\Ii)$ and we obtain this way a natural homomorphism $\pi_{_\Oo}:\GL(W)(\Oo)\to \GL(W_{_\Oo})$. Clearly given any orthonormal basis of $W$, the group $\GL(W)(\Oo)$ is identified with $\GL(m,\Oo)$.

Let $\Qq(W)$ be the vectorspace of $\F$-valued quadratic forms on $W$. As in Section \ref{sec:1} we associate to $f\in \Qq(W)$ the symmetric bilinear form $b_f(\cdot,\cdot)$. We fix a basis $e_1,\ldots,e_m$ of $W$ which is orthonormal for $(\cdot,\cdot)$ and let $(A_f)_{ij}=b_f(e_i,e_j)$ be the associated symmetric matrix. We endow $\Qq(W)$ with the scalar product $(f,g)=\tr(A_fA_g)$. Our next task is to understand the relationship between $\Qq(W_{_\Oo})$ and $\Qq(W)_{_\Oo}$.
\begin{lem}\label{lem:quad}
For a quadratic form $f\in\Qq(W)$ the following are equivalent
 \begin{enumerate}
  \item $f\in \Qq(W)(\Oo)$;
  \item $f(W(\Oo))\subseteq \Oo$;
  \item $b_f(W(\Oo),W(\Oo))\subseteq\Oo$, $b_f(W(\Oo),W(\Ii))\subseteq\Ii$.
 \end{enumerate}
\end{lem}
\begin{proof}
 Clearly $\|f\|^2=\tr (A_f^2)=\sum a_{ij}^2$ belongs to $\Oo$ if and only if $a_{ij}$ belongs to $\Oo$ for all $i,j$ which easily implies the desired equivalences.
\end{proof}
Thus if $f\in \Qq(W)(\Oo)$, then $b_f$ induces a bilinear symmetric form $\ov b_f$ on $W_{_\Oo}\times W_{_\Oo}\to \F_{_\Oo}$ which defines a quadratic form $\ov f\in \Qq(W_{_\Oo})$. 
If $A_f$ is the matrix of $f$ with respect to the orthonormal basis $\{e_1,\ldots,e_m\}$ then the matrix $A_{\ov f}$ representing $\ov f$ with respect to the basis $\{p_{_\Oo}(e_1),\ldots,p_{_\Oo}(e_m)\}$ is just the reduction modulo $\Ii$ of the matrix $A_f$. With this at hand, one verifies easily that the map
$$\begin{array}{cccc}
   \ov p_{_\Oo}&\Qq(W)(\Oo)&\to &\Qq(W_{_\Oo})\\
   &f&\mapsto &\ov f
  \end{array}
$$
induces an isomorphism of $\F_{_\Oo}$ vector spaces
$$\Qq(W)_{_\Oo}\to \Qq(W_{_\Oo}).$$
We end the discussion concerning quadratic forms with the following remark: 
\begin{remark}\label{rem:index}
Let $f\in \Qq(W)$ and $\{e_1\ldots,e_n\}$ be an orthonormal basis in which $f$ is diagonal, that is $b_f(e_i,e_j)=\l_i\delta_{ij}$. Let 
$$\begin{array}{l}
m_f={\rm card}\{i:\l_i>0\},\\
n_f={\rm card}\{i:\l_i<0\},\\
z_f={\rm card}\{i:\l_i=0\}.\\
\end{array}$$
Then clearly $m_{\ov f}\leq m_f$, $n_{\ov f}\leq n_f$, $z_{\ov f}\geq z_f.$
\end{remark}

There is also a reduction process for Grassmannians and it will play an important role for the construction of framings. Thus let $L\in \Gr_l(W)$ be an $l$-dimensional subspace of $W$. Then $L(\Oo)=L\cap W(\Oo)$ and if $e_1,\ldots,e_l$ is an orthonormal basis of $L$ we have $L(\Oo)=\Oo e_1+\ldots+\Oo e_l$. This implies that the image $p_{_\Oo}(L)$ of $L(\Oo)$ in $W_{_\Oo}$ is an $\F_{_\Oo}$-vector subspace of dimension $l$. In this way we obtain a map $q_{_\Oo}:\Gr_l(W)\to \Gr_l(W_{_\Oo})$ which is equivariant with respect to $\pi_{_\Oo}:GL(W)(\Oo)\to \GL(W_{_\Oo})$.

\begin{remark}
The map $q_{_\Oo}$ does not preserve transversality: if $V=\F^2$ with the standard scalar product and $x$ is a non-zero element of $\Ii$ the two distinct lines $\F\cdot(1,0)$ and $\F\cdot(1,x)$ of $\mathbb P V$ have the same image in $\mathbb PV_{_\Oo} $. 
\end{remark}

We apply now the preceding remarks to the following situation. Let $V$ be a $\F$-vector space with a symplectic form $\<\cdot,\cdot\>$ and fix a compatible complex structure $J$. We will use the associated scalar product $(\cdot,\cdot):=\<J\cdot,\cdot\>$ to define the $\Oo$ points. If $L$ is a Lagrangian, then $JL$ is orthogonal to $L$ and if $\{e_1,\ldots,e_n\}$ is an orthonormal basis of $L$, the basis $\Bb=\{e_1,\ldots,e_n,-Je_1,\ldots, -Je_n\}$ is orthonormal and symplectic. With this at hand one shows readily that $J\in \Sp(V)(\Oo):=\Sp(V)\cap\GL(V)(\Oo)$, and that $\<\cdot,\cdot\>$ induces a symplectic form $\<\cdot,\cdot\>_{_\Oo}$ of $V_{_\Oo}$ compatible with $p_{_\Oo}:V(\Oo)\to V_{_\Oo}$. If in addition one sets $J_{_\Oo}=\pi_{_\Oo}(J)$ then $J_{_\Oo}$ is a complex structure on $V_{_\Oo}$ compatible with $\langle\cdot,\cdot\>_{_\Oo}$ and with associated scalar product $(\cdot,\cdot)_{_\Oo}$.
From the above it follows that, if $L\in\Gr_n(V)$ is a Lagrangian, then $q_{_\Oo}(L)\in\Gr_n(V_{_\Oo})$ is a Lagrangian as well. We have
\begin{lem}\label{lem:surj}
The map
$$q_{_\Oo}:\Ll(V)\to\Ll(V_{_\Oo})$$ is surjective.
\end{lem}
\begin{proof}
Let $L_0$ be a $k$-dimensional totally isotropic subspace of $V$ and let $v_0\in V$ be such that $\<v,v_0\>\in \Ii$ for all $v\in L_0$. Let $e_1,\ldots,e_k$ be an orthonormal basis of $L_0$. By completing it to a symplectic basis of $V$ it is easy to verify that the map 
$$\begin{array}{ccc}
V(\Ii)&\to &\Ii^k\\
w&\mapsto&(\<e_1,w\>,\ldots,\<e_k,w\>)
\end{array}$$
is surjective. Thus we can find $w_0\in V(\Ii)$ with $\<e_i,v_0\>=\<e_i,w_0\>$ for all $1\leq i\leq k$. Then $v_1=v_0-w_0$ has the same projection in $V_{_\Oo}$ as $v_0$ and is orthogonal to $L_0$ with respect to the symplectic form. The lemma follows then by recurrence on the dimension.
\end{proof}

\subsection{Affine charts on Lagrangian Grassmannians and reduction modulo $\Ii$}
Now we turn to a more detailed study of the map $q_{_\Oo}$ and certain transversality properties. 
Recall from Section \ref{sec:basic} that given transverse Lagrangians $l_1,l_2$ in $V$ we have a map 
$$j_{l_1,l_2}:\Qq(l_1)\to \Ll(V)^{l_2}$$
which to $f\in \Qq(l_1)$ associates the Lagrangian
$$L_f=\{v+T_fv|\; v\in l_1\}$$
where $T_f:l_1\to l_2$ is defined by the equation
 $$b_f(v,w)= \langle v,T_fw\rangle=\langle w,T_fv\rangle,\quad v,w\in l_1. $$
 
 If $l_1,l_2$ are orthogonal for $(\cdot,\cdot)$ and $\{e_1,\ldots,e_n\}$ is an orthonormal basis of $l_1$ then $J e_1,\ldots,J e_n$ is an orthonormal basis for $l_2$ and the symmetric matrix $A_f$ of $f$ in this basis is given by $(A_f)_{ij}=\<e_i,T_f(e_j)\>=-(Je_i,T_f(e_j))$. Thus it follows from Lemma \ref{lem:quad} that $f$ belongs to $\Qq(l_1)(\Oo)$
 if and only if the matrix coefficients of $T_f$ with respect to the basis $\{e_1,\ldots, e_n\}$ and $\{J e_1,\ldots, J e_n\}$ are in $\Oo$, which in turn is equivalent to $T_f(l_1(\Oo))\subseteq l_2(\Oo)$.

\begin{lem}\label{lem:LEM}
 Let $(l_1,l_2)$ be orthogonal Lagrangians in the symplectic vector space $V$, then $q_{_\Oo}(l_1)$ and $q_{_\Oo}(l_2)$ are orthogonal and the diagram

$$\xymatrix{\Qq(l_1)(\calO )\ar[r]\ar[d]^{\ov p_{_\Oo}}&\Qq(l_1)\ar[r]^{j_{l_1,l_2}}_{\sim}&\Ll(V)^{l_2}\ar[r]&\Ll(V)\ar[d]^{q_{_\Oo}}\\
\Qq(q_{_\Oo}(l_1))\ar[rr]^{j_{q_{_{_\Oo}}(l_1),q_{_{_\Oo}}(l_2)}}_{\sim}&&\Ll(V_{_\Oo})^{q_{_\Oo}(l_2)}\ar[r]&\Ll(V_{_\Oo})
}$$
  commutes. The image under $q_{_\Oo}$ of a Lagrangian that does not belong to $j_{l_1,l_2}(\Qq(l_1)(\Oo))$ is not transverse to $q_{_\Oo}(l_2)$.
\end{lem}

\begin{proof}
Since $l_1$ and $l_2$ are orthogonal, we have for $f\in \Qq(l_1)$:
$$j_{l_1,l_2}(f)(\Oo)=\{v+T_f(v)|\; v\in l_1(\Oo), T_f(v)\in l_2(\Oo)\}.$$

First notice that, if $f$ belongs to $\Qq(l_1)(\Oo)$, then $T_f(l_1(\Oo))$ is contained in $l_2(\Oo)$ and thus we get
$$j_{l_1,l_2}(f)(\Oo)=\{v+T_f(v)|\; v\in l_1(\Oo)\}.$$
Now  $T_f$ induces a well defined map $\ov T_f:q_{_\Oo}(l_1)\to q_{_\Oo}(l_2)$ with the property that 
$$q_{_\Oo}(j_{l_1,l_2}(f))=\{v+\ov T_f(v)|\; v\in q_{_\Oo}(l_1)\}.$$

But $b_f(v,w)$ is by definition equal to $\<v,T_f(w)\>$ and thus $b_{\ov f}(v,w)$ is equal to $\<v,\ov T_f(w)\>_{_\Oo}$ for $v,w\in q_{_\Oo}(l_1)$.
This implies that $j_{q_{_\Oo}(l_1),q_{_\Oo}(l_2)}(\ov p_{_\Oo}(f))$ is equal to $q_{_\Oo}(j_{l_1,l_2}(f))$ and proves the commutativity of the diagram. 

If $f$ does not belong to $\Qq(l_1)(\Oo)$,  we can assume without loss of generality that $T_f(e_1)$ is not in $l_2(\Oo)$. Writing $T_f(e_1)=\sum_{i=1}^n \mu_i J e_i$, let $i_0$ be such that $|\mu_{i_0}|=\max\{|\mu_i|:1\leq i\leq n\}$. Then $\mu_{i_0}$ does not belong to $\Oo$, and hence $\mu=\mu_{i_0}^{-1}$ belongs to $\Ii$. This implies that $T_f(\mu e_1)$ belongs to $l_2(\Oo)$ and its $e_{i_0}$ coordinate is equal to 1. Thus $\mu e_1+T_f(\mu e_1)$ belongs to  $j_{l_1,l_2}(f)(\Oo)$ and 
$$0\neq p_{_\Oo}(\mu e_1+T_f(\mu e_1))\in q_{_\Oo}(j_{l_1,l_2}(f))\cap q_{_\Oo}(l_2).$$
 \end{proof}
\begin{lem}\label{lem:5.6}
Assume $(a,b,c,d)\in \Ll(V)^{(4)}$ is a maximal 4-tuple such that $q_{_\Oo}(a)$ is transverse to $ q_{_\Oo}(b)$, and $q_{_\Oo}(c)$ is transverse to $q_{_\Oo}(d)$. Then for every $x_1\in (\!(b,c)\!)$  and $x_2\in(\!(d,a)\!)$, the subspace $q_{_\Oo}(x_1)$ is transverse to $q_{_\Oo}(x_2)$.
\end{lem}
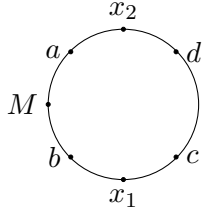
\begin{center}
\begin{tikzpicture}[scale=.5]
\draw (0,0) circle [radius=2];
\filldraw (0,2) node [above] {$x_2$} circle [radius=1.5pt];
\filldraw (1.4,1.4) node [above, right] {$d$} circle [radius=1.5pt];
\filldraw (1.4,-1.4) node [below, right] {$c$} circle [radius=1.5pt];
\filldraw (0,-2) node [below] {$x_1$} circle [radius=1.5pt];
\filldraw (-1.4,-1.4) node [below, left] {$b$} circle [radius=1.5pt];
\filldraw (-2,0) node [left] {$M$} circle [radius=1.5pt];
\filldraw (-1.4,1.4) node [above,left] {$a$} circle [radius=1.5pt];
\end{tikzpicture}
\end{center}
\begin{proof}
Pick $m\in (\!(q_{_\Oo}(a),q_{_\Oo}(b))\!)$ and $M\in \Ll(V)$ with $q_{_\Oo}(M)=m$ (see Lemma \ref{lem:surj}). As a consequence of Remark \ref{rem:index} and the definition of the Kashiwara cocycle we get that $M\in(\!(a,b)\!)$. It follows then that $(b,x_1,c,d,x_2,a)$ forms a maximal 6-tuple and these 6 Lagrangians are all transverse to $M$. Thus these points are in the image of $j_{JM,M}:\Qq(JM)\to \Ll(V)^M$. Denote by $f_l \in \Qq(JM)$ the quadratic form with $j_{JM,M}(f_l)=l\in\Ll(V)^M$. We have from the maximality property of the 6-tuple that $$f_b<\!<f_{x_1}<\!<f_c<\!<f_d<\!<f_{x_2}<\!<f_a$$ (see Lemma \ref{lem:2.10}(2)). Applying now Lemma \ref{lem:LEM} to $l_1=JM$ and $l_2=M$, we deduce from the fact that $q_{_\Oo}(a)$ and $q_{_\Oo}(b)$ are transverse to $m=q_{_\Oo}(M)$ that $f_a,f_b$ are in $Q(JM)(\Oo)$. From the inequalities above we deduce that $f_{x_1}$ and $f_{x_2}$ are in $\Qq(JM)(\Oo)$; it follows then from the commutativity of the diagram in Lemma \ref{lem:LEM} that $q_{_\Oo}(x_1)$ and 
$q_{_\Oo}(x_2)$ are transverse to $m=q_{_\Oo}(M)$.
Also
$$f_{q_{_\Oo}(x_2)}-f_{q_{_\Oo}(x_1)}\geq f_{q_{_\Oo}(d)}-f_{q_{_\Oo}(c)}>\!>0 $$
where the last inequality follows from the hypothesis that $q_{_\Oo}(d)$ is transverse to $q_{_\Oo}(c)$. Thus $q_{_\Oo}(x_2)$ is transverse to $q_{_\Oo}(x_1)$.
\end{proof}

\subsection{Choosing the scale and constructing the maximal framing}

Let $\rho:\G\to\Sp(V)$ be a representation admitting a maximal framing $\phi:S\to \Ll(V)$. 
We assume that there is a complex structure  $J$ in $\mathbb X_V$ and an order convex subring $\Oo$ of $\F$ such that $\rho(\G)\subset \Sp(V)(\Oo)$. We define then $\rho_{_\Oo}:\G\to\Sp(V_{_\Oo})$ as the composition $\rho_{_\Oo}:=\pi_{_\Oo}\circ\rho$ and $\phi_{_\Oo}:S\to\Ll(V_{_\Oo})$ as the composition $\phi_{_\Oo}:=q_{_\Oo}\circ\phi$.
Our goal is to show
\begin{thm}\label{thm:bdrymap}
If $(S,\phi)$ is a maximal framing for $\rho:\G\to\Sp(V)$, then $(S,\phi_{_\Oo})$ is a maximal framing for $\rho_{_\Oo}:\G\to \Sp(V_{_\Oo})$. 
\end{thm}

\begin{remark}
Since $\G$ is finitely generated, for any choice of a compatible complex structure $J$ it is possible to find an infinitesimal $\s$  so that $\rho(\G)\subset \Sp(V)(\Oo_\s)$ where $\Oo_\s$ is the order convex subring  described in Example \ref{ex:4.17}. However, as we will discuss in Section \ref{sec:10}, the choice of $\s$  depends on the complex structure $J$ (see Proposition \ref{prop:int2}). 
\end{remark}
In view of the definition of maximality of triples of Lagrangians and Remark \ref{rem:index}, in order to prove Theorem \ref{thm:bdrymap} we have to show that if $x\neq y$ are distinct points in $S$, then $q_{_\Oo}(\phi(x))$ and $q_{_\Oo}(\phi(y))$ are transverse Lagrangians.
As a first step we show
\begin{lem}\label{lem:5.8}
 Assume that there exist two distinct points $x,y$ in $S$ such that $q_{_\Oo}(\phi(x))$ and $q_{_\Oo}(\phi(y))$ are not transverse. Then there exists a hyperbolic element $\g\in\G$ such that $q_{_\Oo}(\phi(\gamma^+))$ and $q_{_\Oo}(\phi(\g^-))$ are not transverse.
\end{lem}
\begin{proof}
From Lemma \ref{lem:5.6} it follows that for $I$ either $(\!(x,y)\!)$ or $(\!(y,x)\!)$ we have that for every $t_1,t_2$ in $I$, $q_{_\Oo}(\phi(t_1))$ and $q_{_\Oo}(\phi(t_2))$ are not transverse. Now pick a hyperbolic element $\g\in \G$ with $\{\g^+,\g^-\}\subset I$.
\end{proof}

The strategy of the proof consists in showing that  for every hyperbolic element $\g\in\G$, the Lagrangians $q_{_\Oo}(\phi(\g^-))$ and $q_{_\Oo}(\phi(\g^+))$ are transverse. This will be a consequence of the properties of the eigenvalues of $\rho(\g)$ using the Collar Lemma. 

We first observe that eigenvalues behave well with respect to reduction modulo $\Ii$:
\begin{lem}\label{lem:c}
 Let $B\in \GL_m(\Oo)$ be a matrix, and denote by $\beta_i\in\K$ the eigenvalues of $B$. Then
 \begin{enumerate}
  \item $|\beta_i|\in \Oo$;
  \item if $\ov B$ denotes the image of $B$ in $\GL_m(\F_{_\Oo})$, and $\ov \beta_i$ are the images of $\beta_i$ in $\K_{_\Oo}$, then the eigenvalues of $\ov B$ are precisely $\ov \beta_i$.
 \end{enumerate}
\end{lem}

\begin{proof}
The first assertion follows from the fact that if $\beta_i$ is an eigenvalue of $B$ then there exists a vector $v\in V(\Oo)\setminus V(\Ii)$ such that $\|Bv\|=|\beta_i|\|v\|$ (cfr. Lemma \ref{lem:ab}). The second assertion follows from the fact that the characteristic polynomial of the reduction $\ov B$ is the reduction of the characteristic polynomial of $B$.
\end{proof}
\begin{rem}
Clearly, if $g$ belongs to $\GL(V)(\Oo)$, for each subspace $W$ of $V$ preserved by $g$ the restriction $g|_W$ belongs to $\GL(W)(\Oo)$ and the restriction commutes with the reduction: $\pi_{_\Oo}(g)|_{q_{_\Oo}(W)}=\pi_{_\Oo}(g|_{W})$.
However it is worth pointing out that the Jordan decomposition of a matrix $B\in \GL_m(\Oo)$ is  not necessarily defined in $\GL_m(\Oo)$ and in particular the exponents of the minimal polynomial of a matrix $B$ need not to be related with the exponents of the minimal polynomial of the reduction of $B$. For example, if $\epsilon$ belongs to $\Ii$, then  the reduction of the not diagonalizable matrix $\bsm2&\epsilon\\0&2\esm$ is diagonalizable and the reduction of the diagonalizable matrix $\bsm 1&1 \\0& 1+\epsilon\esm$ is not diagonalizable.
\end{rem}
This last example shows that generalized eigenspaces relative to distinct eigenvalues might not have transverse images in the quotient if the corresponding eigenvalues coincide modulo $\Ii$. We will now deduce from the Collar Lemma that, in case of  framed maximal representations, the intermediate eigenvalues have distinct reductions:
\begin{lem}\label{cor:1}
Let $\rho:\G\to\Sp(V)$ be a representation admitting a maximal framing. Assume that $\rho(\G)\subset\Sp(V)(\Oo)$. Then for every hyperbolic element $\g\in\G$, we have 
$$|\l_n(\g)|-1\in \Oo\setminus \Ii$$
where $|\l_1(\g)|\geq\ldots\geq|\l_n(\g)|>1$ are the eigenvalues of $\g$ of absolute value greater than 1.
 \end{lem}
 \begin{proof}
Let $\delta\in\G$ be a hyperbolic element with positive intersection number with $\g$ and let $\l_1(\delta)$ be the eigenvalue of $\rho(\delta)$ of largest modulus. If $|\l_n(\g)|<2$, then the Collar Lemma (Theorem \ref{lem:collar}) implies 
$$|\l_n(\g)|-1=\frac{|\l_n(\g)|^{2}-1}{|\l_n(\g)|+1}\geq \frac 1{3|\l_1(\delta)|^{2n}}.$$
Now observe that, since $\rho(\delta)\in \Sp(2n,\Oo)$, we have that $|\l_1(\delta)|$ belongs to $\Oo$ from which the claim follows.
\end{proof}

 We have now all the necessary ingredients to prove Theorem \ref{thm:bdrymap}:
 \begin{proof}[{Proof of Theorem \ref{thm:bdrymap}}]
Let us assume by contradiction that there exist $x,y$ in $S^1$ with $q_{_\Oo}(\phi(x))$ non transverse to $q_{_\Oo}(\phi(y))$. As a consequence of Lemma \ref{lem:5.8} we can find a hyperbolic element $\g$ in $\G$  such that  $q_{_\Oo}(\phi(\g^+))$ is non-transverse to $q_{_\Oo}(\phi(\g^-))$. 
 
If now $|\l_1(\g)|\geq\ldots\geq |\l_n(\g)|>1$ are the absolute values of the eigenvalues of $\rho(\g)|_{\phi(\g^+)}$, counted with multiplicity, then it follows from Lemma \ref{cor:1} and \ref{lem:c} that the absolute values $|\ov{\l_1(\g)}|\geq\ldots\geq |\ov{\l_n(\g)}|>1$ of the eigenvalues of the restriction of $\rho_{_\Oo}(\g)$ to $q_{_\Oo}(\phi(\g^+))$ are all strictly larger than 1. Since $|\ov{\l_1(\g)}|^{-1}\leq\ldots\leq |\ov{\l_n(\g)}|^{-1}<1$ are then the absolute values of the eigenvalues of the restriction of $\rho_{_\Oo}(\g)$ to $q_{_\Oo}(\phi(\g^-))$, this implies that the $\rho_{_\Oo}$-invariant vectorspace $q_{_\Oo}(\phi(\g^+))\cap q_{_\Oo}(\phi(\g^-))$ must be zero since otherwise $\rho_{_\Oo}(\g)$ would have at least a non-zero eigenvalue which would be an element in $\K_{_\Oo}$ both of absolute value strictly larger and smaller than 1. Thus $q_{_\Oo}(\phi(\g^+))\cap q_{_\Oo}(\phi(\g^-))=0$ which is a contradiction. Hence, for every $x\neq y$ in $S^1$,  $q_{_\Oo}(\phi(x))$ is transverse 
to 
$q_{_\Oo}(\phi(y))$.
\end{proof}

\section{Fields with valuation and the projection to the building}\label{sec:building}
In this section $\F$ will denote an ordered field with a compatible valuation $v:\F\to\R\cup\{\infty\}$, meaning that we require that whenever $0\leq x\leq y$ we have $v(y)\leq v(x)$.

\begin{example}[Cfr. Example \ref{ex:4.17}]
Let $\mathbb E$ be an ordered field, $\s\in \mathbb E$ be an infinitesimal and $\Oo_\sigma$ the order convex local subring consisting of elements comparable with $\s$. 
On $\Oo_\s$ we define the valuation
$$v_\s(x)=\sup\{t\in \R| \;|x|\leq \s^t\}.$$
Then $v_\s$ passes to the quotient $\mathbb E_\s:=\Oo_\s/\Ii_\s$ by the maximal ideal $\Ii_\s$ and defines an order compatible valuation.
\end{example}
\smallskip


We introduce on $\F$ the norm $\|x\|:=e^{-v(x)}.$
This defines an ultrametric norm on $\F$ with valuation ring $\Uu:=\{x\in\F|\;\|x\|\leq 1\}$
whose maximal ideal is
$\Mm:=\{x\in\F|\;\|x\|< 1\}.$ Observe that since the valuation is order compatible, the norm is order compatible as well: if $0<x<y$ then $\|x\|\leq \|y\|$.

Let $(V,\<\cdot,\cdot\>)$ be a symplectic vector space over $\F$, $J_0\in \mathbb X_V$ a compatible complex structure and $(\cdot,\cdot)_{J_0}$ the corresponding scalar product. We denote by $\Bb_V$ the affine building associated to $\Sp(V)$ (see \cite[Section 3.2]{APbuil} and \cite[Theorem 4.3]{KT1}). It is well known that the set of vertices $\Bb_V^0$ of $\Bb_V$ can be identified with the homogeneous space $\Sp(V)/\Sp(V)(\Uu)$ where we define as in Section \ref{sec:4}
$$V(\Uu)=\{v\in V|\; (v,v)\in \Uu \}$$
and  
$$\Sp(V)(\Uu)=\{g\in\Sp(V)|g(V(\Uu))=V(\Uu)\}.$$
The stabilizer of the complex structure $J_0\in \mathbb X_V$ is 
$$\begin{array}{rl}
U(J_0)&=\{g\in\Sp(V)|\; gJ_0 g^{-1}= J_0\}\\
&=\{g\in\Sp(V)|\text{ $g$ preserves the scalar product $(\cdot,\cdot)_{J_0}$}\}
\end{array}$$
and hence is contained in $\Sp(V)(\Uu)$. As a result we can define the projection
$$\pi_\Bb:\mathbb X_{V}=\Sp(V)/U(J_0)\to \Bb_V^0=\Sp(V)/\Sp(V)(\Uu).$$
\begin{remark}
Parreau gave an explicit description of the building associated to $\SL(2n,\F)$ as the space of good norms on $\F^{2n}$ of determinant one \cite{APbuil}. It is possible to verify that, considering the affine building associated to $\Sp(2n,\F)$ as a subbuilding of the affine building associated to $\SL(2n,\F)$, the map $\pi_\Bb$ corresponds to the map that associates to a point $J\in\mathbb X_V$ the corresponding good norm $\eta_J(v)=\|(v,v)_J\|$.
\end{remark}
For $\F=\R$ Siegel gave explicit formulas for the Riemannian distance on $\Xx_\R$ \cite{Siegel}. We use the crossratio $R$ defined in Section \ref{sec:cr} to define in our context a distance like function as follows.
Observe that, given $X,W\in \Tt_V$, the crossratio $R(X,\s(W),W,\s(X))$ is always well defined: 
indeed, the Hermitian form $i\<\cdot,\s(\cdot)\>$ is positive definite on $X$ and $W$ and negative definite on $\s(W)$ and $\s(X)$, in particular $X$ and $\s(W)$ are transverse and so are $W$ and $\s(X)$.
Moreover all the eigenvalues of the crossratio $R(X,\s(W),W,\s(X))$ belong to $\F$ and are between $0$ and $1$: indeed since $\F$ is real closed, for each pair $X,W\in\Tt_V$ we can find $g\in\Sp(V)$ such that $g_*X=i\Id$, $g_*W=iD$ for a diagonal matrix $D$ with positive entries, and we have
$$gR(X,\s(W),W,\s(X))g^{-1}=R(i\Id,-iD,iD,-i\Id)=\frac{(\Id-D)^2}{(\Id+D)^2}.$$

We can thus define
\begin{equation}\label{eqn:eqn} d(Z,W)=\sqrt{\sum_{i=1}^n\left(\ln\left\|\frac{1+\sqrt r_i}{1-\sqrt r_i}\right\|\right)^2}\end{equation}
where $r_1, \ldots, r_n$ are the eigenvalues of $R(X,\s(W),W,\s(X))$.

In the case we considered above, where  $X=i\Id$ and $W=iD$,  Equation (\ref{eqn:eqn}) specializes to: 
$$d(i\Id,i D)=\sqrt{\sum_{i=1}^n\left(\ln\left\|d_i\right\|\right)^2}$$
where $d_1,\ldots,d_n$ are the entries of $D$.

The function $d$ is clearly $\Sp(V)$ invariant since the eigenvalues of the crossratio are. Denote by $d_{\Bb}$ the CAT(0) distance on $\Bb_V$. Using the  transitivity of the symplectic group on apartments in $\Bb_V$ and the invariance of $d$ one verifies:
\begin{prop}\label{prop:proj}
For any $X,Y\in \Tt_{V}$ we have
$$d_{\Bb}(\pi_\Bb(X),\pi_\Bb(Y))=d(X,Y).$$
\end{prop}
As a result, we get that $d$ is a pseudodistance on $\Tt_V$ and $\Bb_V$ is the Hausdorff quotient of $\Tt_V$ modulo this pseudodistance.

We will denote by $L_\Bb(g)$ the translation length of an element $g\in \Sp(V)$ considered as an isometry of the affine building $\Bb_V$. 

\section{On elements with fixed points}\label{sec:fixedpoint}
 We place ourselves in the framework of Section \ref{sec:building} and consider a representation $\rho:\G\to \Sp(V)$ admitting a maximal framing $(S,\phi)$. In this section we want to analyze how elements of $\G$ which have zero translation length in the building $\Bb_V$ interact. As a crucial step in the analysis, we associate to any such $\g\in \G$  a pair $(b^+_\g,b^-_\g)$ of points in $\Bb_V$ which are fixed by $\rho(\g)$ and are canonically constructed from the maximal framing $\phi$.  

Recall from Section \ref{sec:building} that we denote by $\pi_\Bb:\Tt_{V}\to\Bb_V$ the $\Sp(V)$ equivariant projection from the Siegel upper half-space to the affine building associated to $\Sp(V)$, and, given an element $g\in \Sp(V)$, we denote by $L_{\Bb}(g)$ the translation length of $g$ on $\Bb_V$. Moreover, for ease of notation, we will denote by $\Yy_\g$ the $\F$-tube $\Yy_{\phi(\g^-),\phi(\g^+)}$ and by $\YY_\g$ its projection to $\Bb_V$: 
$$\YY_\g=\pi_\Bb(\Yy_\g).$$
It follows from the equivariance of $\pi_\Bb$ that $\YY_\g$ is a subbuilding of $\Bb_V$ associated to a subgroup of $\Sp(V)$ isomorphic to $\GL_n(\F)$.
Recall from Section \ref{sec:ort} that given any pair of transverse Lagrangians $a,b\in \Ll(V)$, we defined an orthogonal projection
$$\pr_{\Yy_{a,b}}:(\!(a,b)\!)\cup (\!(b,a)\!)\to \Yy_{a,b}.$$

We will prove 
\begin{prop}\label{prop:lazy}
 Let $\g\in \G$ be an element which is not boundary parallel. 
Assume that $L_{\Bb}(\rho(\g))=0$. Then both maps
$$\begin{array}{cccc}
  F_\g^+:& (\!(\g^-,\g^+)\!)&\to &\YY_\g\\
   &x&\mapsto&\pi_\Bb(\pr_{\Yy_\g}(\phi(x)))
  \end{array}
$$
and $$\begin{array}{cccc}
   F_\g^-:&(\!(\g^+,\g^-)\!)&\to &\YY_\g\\
   &x&\mapsto&\pi_\Bb(\pr_{\Yy_\g}(\phi(x)))
  \end{array}
$$
are constant.
\end{prop}
Denoting by $b_\g^+$ resp. $b_\g^-$ the constant images of the maps $F^\pm_\g$ in Proposition \ref{prop:lazy} we have
\begin{cor}
 The points $b^+_\g$ and $b^-_\g$ are fixed by $\rho(\g)$.
\end{cor}
If $\g\in\G$ corresponds to a simple closed geodesic, it is possible to construct examples of representations $\rho:\G\to\Sp(V)$ such that the points $b_\g^+$ and $b_\g^-$ are different. The second main result of the section gives sufficient conditions for the two points to coincide:
\begin{prop}\label{prop:intersectinglazy}
 Assume that $\g$ and $\eta$ in $\G$ are hyperbolic elements with intersecting axes and $L_\Bb(\rho(\g))=L_\Bb(\rho(\eta))=0$.
 Then $$b^+_\g=b^-_\g=b^+_\eta=b^-_\eta=\pi_\Bb\left(\Yy_{\phi(\g^+),\phi(\g^-)}\cap \Yy_{\phi(\eta^+),\phi(\eta^-)}\right).$$
\end{prop}
\begin{cor}
Assume that $L_\Bb(\rho(\g))=0$. If the closed geodesic corresponding to $\g$ is not simple, then $b^+_\g=b^-_\g$.
\end{cor}
Before proceeding to the proof of Proposition \ref{prop:lazy} and Proposition \ref{prop:intersectinglazy} we observe that in certain situations one can get a uniform lower bound on the translation lengths $L_\Bb(\rho(\g))$ for all hyperbolic elements $\g$ crossing a given hyperbolic element $\eta$. This is in fact an immediate corollary of the Collar Lemma:
\begin{cor}\label{cor:7.5}
Assume that $\eta\in\G$ is a hyperbolic element and let us denote by $|\l_1(\eta)|\geq\ldots\geq |\l_n(\eta)|>1$ the eigenvalues of $\eta$ of absolute value larger than 1. If $\delta=\||\l_n(\eta)|-1\|<1$, then for any element $\g$ having positive intersection number with $\eta$ we have $$L_\Bb(\rho(\g))\geq \frac{1}{2n\delta}.$$ In particular if the closed geodesic represented by $\eta$ is not simple, then $\||\l_n(\eta)|-1\|\geq 1$.
\end{cor}
Proposition \ref{prop:intersectinglazy} also allows us to give sufficient conditions for a representation $\rho$ to have a global fixed point.
We say that a generating set $X$ for $\G$ is connected if the graph $(X,E)$, where $E$ consists of the pairs $(s_1,s_2)$ of elements of $X$ whose axis intersect, is connected.
\begin{cor}\label{cor:3}
 Let $X$ be any connected generating set for $\G$.  If $\rho:\G\to \Sp(V)$ is a  representation admitting maximal framing the following are equivalent:
 \begin{enumerate}
  \item $\rho$ has a global fixed point in $\Bb_{V}$;
  \item $L_{\Bb}(\rho(s))=0$ for all $s\in X$.
 \end{enumerate}
\end{cor}
\begin{remark} 
 There exist connected generating sets consisting of $2g$ simple closed curves. In particular Corollary \ref{cor:3} refines, in our setting \cite[Corollary 3]{APcomp}. 
\end{remark}

Recall from Section \ref{sec:2.2} that we say that $g\in\Sp(V)$ is Shilov hyperbolic if there exists a $g$-invariant decomposition $V=L^+_g\oplus L^-_g$ such that all the eigenvalues of the restriction $M_g$ of $g$ to $L_g^+$ are in absolute value strictly greater than one. It is however worth remarking that in general $g$ does not necessarily have a hyperbolic dynamic on $\Ll(V)$. 
It follows from Corollary \ref{cor:1} that, as soon as $\rho$ admits a maximal framing, for any hyperbolic element $\g\in\G$, its image $\rho(\g)$ is Shilov hyperbolic.

\begin{lem}\label{lem:2.28}
 Let $g\in\Sp(V)$ be Shilov hyperbolic, and let $\{\l_1,\ldots,\l_n\}\subset\F[i]$ be the set of eigenvalues of $M_g$. Then
 $$L_{\Bb}(g)=2\sqrt{\sum_{i=1}^n(\ln\|\l_i\|)^2 }$$
\end{lem}
\begin{proof}
 Since $g$ is Shilov hyperbolic it stabilizes the $\F$-tube $\Yy_{L_g^+,L_g^-}$, and similarly it stabilizes the projection $$\mathbb Y_{L_g^+,L_g^-}=\pi_\Bb(\Yy_{L_g^+,L_g^-}).$$
 This latter is a subbuilding of $\Bb_V$ associated to $\GL(n,\F)$. The desired statement then follows from \cite{APbuil}
\end{proof}

\begin{lem}\label{lem:2.29}
Let $g\in \Sp(V)$ be Shilov hyperbolic. Then the following are equivalent
\begin{enumerate}
 \item $L_{\Bb}(g)=0$;
 \item $\|\det M_g\|=1$;
 \item $\|\det R(L^+_g,S,gS,L^-_g)\|=1$ for every $S$ in $(\!(L^+_g,L^-_g)\!)$.
\end{enumerate}
\end{lem}
\begin{proof}
In view of Lemma \ref{lem:2.28}, we have that 
$$L_{\Bb}(g)=2\sqrt{\sum_{i=1}^n(\ln\|\l_i\|)^2 }$$
while $$\|\det M_g\|=\prod_{i=1}^n\|\l_i\|$$ and
$$\det R(L^+_g,S,gS,L^-_g)=(\det M_g)^2$$ 
The equivalence follows easily from the assumption that $|\l_i|>1$ for all $i$ and the order compatibility of the norm.
\end{proof}

\begin{lem}\label{lem:8.2}
 Let us assume that the 5-tuple of Lagrangians $(x_1,x_2,x_3,x_4,x_5)$ is maximal then 
 $$\det R(x_1,x_2,x_3,x_5)\leq \det R(x_1,x_2,x_4,x_5).$$
\end{lem}
\begin{proof}
 We may assume that $x_1=0$ and $x_5=l_\infty$, then we have $0<\!<x_2<\!<x_3<\!<x_4$. In this case a computation gives that $R(x_1,x_2,x_3,x_5)$ is conjugate to $y_1=x_2^{-1/2}x_3x_2^{-1/2}$ and $R(x_1,x_2,x_4,x_5)$ is conjugate to $y_2=x_2^{-1/2}x_4x_2^{-1/2}$. Since each eigenvalue of  $y_1$ is positive and smaller than the corresponding eigenvalue of $y_2$ one obtains the desired inequality.
\end{proof}

\begin{lem}\label{lem:8.6}
Assume that $(a,x,y,b)$ in $\Ll(V)^4$ is maximal. Then 
\begin{enumerate}
\item $\|\det R(a,x,y,b)\|\geq 1.$
\item ${d}(\pr_{\Yy_{a,b}}(x),\pr_{\Yy_{a,b}}(y))\leq \ln\|\det R(a,x,y,b)\|\leq \sqrt n ~{d}(\pr_{\Yy_{a,b}}(x),\pr_{\Yy_{a,b}}(y)).$
\end{enumerate}
\end{lem}
\begin{proof}
Since $\Sp(V)$ is transitive on maximal triples we can assume that $a=0,b=l_\infty$ and $x$ corresponds to the matrix $+\Id$. Since the triple $(x,y,l_\infty)$ is maximal $y$ corresponds to a positive definite matrix with all eigenvalues strictly bigger than one.  The first statement is immediate since $\det R(a,x,y,b)=\det(Y)$.

Now it follows from  the definition of the orthogonal projection that $\pr_{\Yy_{a,b}}(x)=i\Id$ and $\pr_{\Yy_{a,b}}(y)=iY$. 
If $\l_1,\ldots,\l_n$ are the eigenvalues of $Y$, the explicit formula for the distance $d$ gives
$$d(i\Id,iY)=\sqrt{\sum_{i=1}^n(\ln\|\l_i\|)^2}$$
and we have
$$\ln\|\det R(a,x,y,b)\|=\sum_{i=1}^n\ln(\|\l_i\|).$$
The second assertion in the lemma then follows from Cauchy-Schwartz and the fact that $\ln\|\l_i\|\geq 0$ for every $i$.
\end{proof}

\begin{lem}\label{lem:8.7}
 Assume $L_\Bb(\rho(\g))=0$ then for any $x,y\in (\!(\g^-,\g^+ )\!) $ with $(\g^-,x,y,\g^+)$ positively oriented we have
 $$\|\det R(\phi(\g^-),\phi(x),\phi(y),\phi(\g^+))\|=1$$
\end{lem}
\begin{proof}
Since $(\g^-,x,y,\g^+)$ is positively oriented, and $\g^+$ is the attractive fixed point of $\g$, we can pick $n\geq 1$ with $(x,y,\g^nx)$ positively oriented. Then by Lemma \ref{lem:8.2} we have
$$\begin{array}{rl}
1\leq&\det(R(\phi(\g^-),\phi(x),\phi(y),\phi(\g^+)))\\
\leq&\det(R(\phi(\g^-),\phi(x),\rho(\g)^n\phi(x),\phi(\g^+)))
\end{array}$$
and the latter has norm 1 by Lemma \ref{lem:2.29} (3).
\end{proof}
\begin{proof}[Proof of Proposition \ref{prop:intersectinglazy}]
Let $s,t$ be points in $ (\!(\g^-,\g^+)\!)$ and assume without loss of generality that the quadruple $(\g^-,t,s,\g^+)$ is positively oriented. Then $(\phi(\g^-),\phi(t),\phi(s),\phi(\g^+))$ is a maximal quadruple, thus by Lemma \ref{lem:8.6} we have 
${ d}(\pr_{\Yy_{a,b}}(x),\pr_{\Yy_{a,b}}(y))\leq \ln\|\det R(a,x,y,b)\|.$
The right hand side vanishes by Lemma \ref{lem:8.7}, and hence we obtain, using Proposition \ref{prop:proj}, that $\pi_\Bb(\pr_{\Yy_{\g}}(\phi(x)))=\pi_\Bb(\pr_{\Yy_{\g}}(\phi(y)))$.
\end{proof}

Let us now assume that there are two elements $\g,\eta$ in $\pi_1(\S)$ whose axes intersect. We want to show that if both $\rho(\g)$ and $\rho(\eta)$ fix a point in $\Bb_V$, then they share a fixed point. We begin with a preliminary computation:
\begin{lem}\label{lem:8.12}
Let $\g$ and $\eta$ be two hyperbolic elements of $\G$ with intersecting axis. Assume $L_{\Bb}(\rho(\g))=L_{\Bb}(\rho(\eta))=0$ and that the quadruple $(\eta^-,\g^-,\eta^+,\gamma^+)$ is positively oriented.  Then  for every $ x\in (\!(\g^-,\g^+)\!)$ all eigenvalues of the crossratio $R(\phi(\eta^-),\phi(\g^-),\phi(x),\phi(\g^+))$ have the form $1+f$ where $f\in\F^{>0}$ satisfies $\|f\|=1$.
\end{lem}

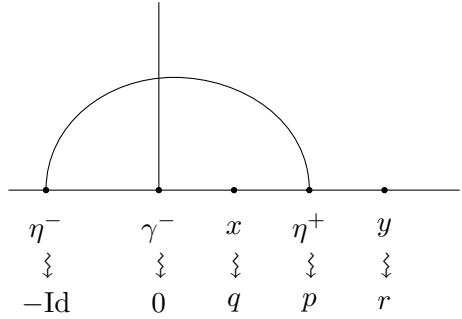
\begin{center}
\begin{tikzpicture}
\draw (-2,0) to (4,0);
\draw (0,0) to (0,2.5);
\draw (-1.5,0) [in=180, out=90] to (0.2,1.5) [in=90, out=0] to (2,0) ;

\filldraw  (-1.5,0) circle [radius=1pt];
\filldraw  (0,0) circle [radius=1pt];
\filldraw  (1,0) circle [radius=1pt];
\filldraw  (2,0) circle [radius=1pt];
\filldraw  (3,0) circle [radius=1pt];

\node at (-1.5,-.5) {$\eta^-$};
\node at (0,-0.5) {$\gamma^-$};
\node at (1,-0.5) {$x$};
\node at (2,-0.5) {$\eta^+$};
\node at (3,-0.5) {$y$};

\node at (-1.5,-1)[rotate=-90] {$\rightsquigarrow$};
\node at (0,-1) [rotate=-90] {$\rightsquigarrow$};
\node at (1,-1) [rotate=-90] {$\rightsquigarrow$};
\node at (2,-1) [rotate=-90] {$\rightsquigarrow$};
\node at (3,-1) [rotate=-90] {$\rightsquigarrow$};

\node at (-1.5,-1.5) {$-\Id$};
\node at (0,-1.5) {$0$};
\node at (1,-1.5) {$q$};
\node at (2,-1.5) {$p$};
\node at (3,-1.5) {$r$};
\end{tikzpicture}
\end{center}
\begin{proof}
Pick $g\in\Sp(V)$ such that $g_*(\phi(\eta^-),\phi(\g^-),\phi(\g^+))=(-\Id,0,l_\infty)$ and set $p=g_*(\phi(\eta^+))$.  Now pick $x\in(\!(\g^-,\eta^+)\!)$ and set $q=g_*(\phi(x))$. Observe that $0<\!<q<\!<p$.

By Lemma \ref{lem:8.7}, since $L_{\Bb}(\rho(\g))=0$, we have 
$$\|\det R(\phi(\g^-),\phi(x),\phi(\eta^+),\phi(\g^+))\|=1$$ 
which implies $\|\det p\|=\|\det q\|$.

Let $\mu_1\geq\ldots\geq\mu_n>0$ and $\l_1\geq\ldots\geq\l_n>0$ denote the eigenvalues of $q$ and $p$ respectively. Since $0<\!<q<\!<p$, we deduce that $0<\mu_i<\l_i$ and hence $\|\mu_i\|\leq\|\l_i\|$. This implies that $\|\mu_i\|=\|\l_i\|$ since we know that their products are equal.

Exploiting that $L_\Bb(\rho(\eta))=0$ together with Lemma \ref{lem:8.7} we get 
$$\|\det R(\phi(\eta^-),\phi(\gamma^-),\phi(x),\phi(\eta^+))\|=1$$
which implies that $\|(\det p)(\det (p-q))^{-1}\det(\Id+q)\|=1$. 
From this we deduce
$$\prod_{i=1}^n\|1+\mu_i\|=\|\det(\Id+q)\|=\left\|\frac{\det(p-q)}{\det p}\right\|\leq 1$$
where the last inequality follows from $0<\!<p-q<\!<p$. Together with the observation that $1+\mu_i\geq 1$ and the ultrametric inequality this implies $\|\mu_i\|\leq 1$ for all $i$ and thus $\|\l_i\|=\|\mu_i\|\leq 1$.

Now let $y\in(\eta^+,\g^+)$ and set $r=g_*(\phi(y))$. Then $0<\!<p<\!<r$. Again by Lemma \ref{lem:8.7} we deduce that $$\|\det R(\phi(\g^-),\phi(\eta^+),\phi(y),\phi(\g^+))\|=1$$ 
which implies $\|\det p\|=\|\det r\|$.

Let $\nu_1\geq\ldots\geq\nu_n>0$ denote the eigenvalues of $r$. Since $p<\!<r$ we deduce that $0<\l_i<\nu_i$ and hence $\|\nu_i\|\geq\|\l_i\|$. This implies, as above, that $\|\nu_i\|=\|\l_i\|$.
Since $L_\Bb(\rho(\eta))=0$, Lemma \ref{lem:8.7} implies that 
$$\|\det R(\phi(\eta^+),\phi(y),\phi(\gamma^+),\phi(\eta^-))\|=1$$
that is $\|\det (\Id+r)\|=\|\det(r-p)\|$. 
Since $0<\!<r-p<\!<r$, we obtain $\|\det(r-p)\|\leq \|\det r\|$. On the other hand $0<\!<r<\!<\Id+r$ and hence $\|\det (\Id+r)\|=\|\det (r)\|$ or  equivalently $\prod_{i=1}^n \|1+\frac 1{\nu_i}\|=1.$ This together with the information that $\nu_i>0$ and the ultrametric inequality implies $\|\nu_i\|\geq 1$ and thus $\|\l_i\|=\|\nu_i\|\geq 1$.

To conclude the proof we observe that $R(\phi(\eta^-),\phi(\g^-),\phi(x),\phi(\g^+))$ is conjugate to $R(-\Id,0,q,l_\infty)=\Id+q$ and hence has as all eigenvalues of the form $1+f$ with $f$ positive satisfying $\|f\|=1$.
\end{proof}
\begin{remark}
Recall from Definition \ref{defn:4.12} that $\Yy_\g$ and $\Yy_\eta$ are orthogonal if and only if $R(\phi(\eta^-),\phi(\g^-),\phi(x),\phi(\g^+))=2\Id$. Lemma \ref{lem:8.12} should be interpreted as a weaker form of orthogonality for the projections $\mathbb Y_\g$ and $\mathbb Y_\eta$.
\end{remark}
\begin{lem}\label{lem:8.14}
Let $(a,c,b,d)\in \Ll(V)^4$ be a maximal quadruple and assume that all the eigenvalues of $R(a,c,b,d)$ have the form $1+f$ for some $f\in\F^{>0}$ with $\|f\|=1$. Then the points $$\pr_{\Yy_{c,d}}(a),\;\pr_{\Yy_{c,d}}(b),\;\pr_{\Yy_{a,b}}(c),\;\pr_{\Yy_{a,b}}(d),\;\Yy_{a,b}\cap\Yy_{c,d}$$ have pairwise pseudodistance zero.
\end{lem}
\begin{proof}
Pick $g\in \Sp(V)$ such that $g_*(a,c,b,d)=(-\Id,0,D,l_\infty)$ where $D$ is diagonal with strictly positive entries. Then a computation gives $\pr_{\Yy_{0,l_\infty}}(-\Id)=i\Id$, $\pr_{\Yy_{0,l_\infty}}(D)=iD$ and $\Yy_{0,l_\infty}\cap \Yy_{-\Id,D}=i\sqrt D$.

Now since $D=\diag(d_1,\ldots,d_n)$ the assumption on the eigenvalues implies $\|d_i\|=1$ and the explicit formula for the distance gives the desired statement. 
\end{proof}
\begin{proof}[Proof of Proposition \ref{prop:intersectinglazy}]
We may assume that $(\eta^-,\g^-,\eta^+,\gamma^+)$ is positively oriented. Applying Lemma \ref{lem:8.12} to $x=\eta^+$ we obtain that the pseudodistances of the points $\pr_{\Yy_{\g}}(\phi(\eta^+))$, $\pr_{\Yy_{\g}}(\phi(\eta^-))$, $\pr_{\Yy_{\eta}}(\phi(\g^+))$, $\pr_{\Yy_{\eta}}(\phi(\g^-))$, $\Yy_{\g}\cap\Yy_{\eta}$ are all zero. This concludes the proof once one notices that  (see Proposition \ref{prop:lazy}) 
$$\begin{array}{l}
b^+_\g=\pi_\Bb(\pr_{\Yy_{\g}}(\phi(\eta^+)))\\
b^-_\g=\pi_\Bb(\pr_{\Yy_{\g}}(\phi(\eta^-)))\\
b^+_\eta=\pi_\Bb(\pr_{\Yy_{\eta}}(\phi(\g^-)))\\
b^-_\eta=\pi_\Bb(\pr_{\Yy_{\eta}}(\phi(\g^+))).
\end{array}$$
\end{proof}

\section{Decomposition Theorem}\label{sec:dec}
Let $\rho:\pi_1(\S,x)\to \Sp(V)$ be a representation into a symplectic group over a real closed field $\F$ with valuation,
 and let $\pi_{\Bb}:\Tt_V\to\Bb_V$ denote the projection to the building. Recall from the introduction that
if  $\S=\bigcup_{v\in\calV}\S_v$ is a decomposition of the surface $\S$ into subsurfaces with geodesic boundary, we consider the associated presentation of $\G$ as fundamental group of a graph of groups with vertex set $\calV$ and vertex groups $\pi_1(\S_v)$. We denote by $\wt\calV$ the vertex set of the associated Bass-Serre tree $\calT$. For every $v\in\calV$ and $w\in\wt\calV$ lying above $v$, the stabilizer $\G_w$ of $w$ in $
G$ is isomorphic to $\pi_1(\S_v)$. In this section we prove the result mentioned in the introduction as Theorem \ref{thm:6}:
\begin{thm}\label{thm:dec}
Assume that $\rho:\G\to\Sp(V)$ admits a maximal framing. Then there is a decomposition $\S=\bigcup_{v\in V}\Sigma_v$ of $\S$ into subsurfaces with geodesic boundary such that
\begin{enumerate}
\item for every $\g\in\G$ whose associated closed geodesic is not contained in any subsurface,  $L_\Bb(\rho(\g))>0$;
\item for every $v\in \calV$  there is the following dichotomy:
\begin{enumerate}
\item[(A)] for every $w\in\wt \calV$ lying above $v$, and any $\g\in\G_w$ which is not boundary parallel,  $L_\Bb(\rho(\g))>0$,
\item[(B)]  for every $w\in\wt \calV$ lying above $v$, there is a point $b_w\in\Bb_V$ which is fixed by $\G_w$.
\end{enumerate}
\end{enumerate}
\end{thm}

The proof of the theorem is based on the analysis of the incidence structure of the set
$$\Ll_\rho=\{\g\in\G|\; \g\neq e, \g \text{ hyperbolic, } L_\Bb(\rho(\g))=0\}.$$
Let 
$$\P\Ll_\rho=\{\g\in\Ll_\rho|\; \g \text{ is primitive}\}/_{\g\sim\g^{-1}}$$
and denote by $\ov \g\in \P\Ll_\rho$  the equivalence class of $\g$.
Let
$$\Aa_\rho=\{\ax(\g)|\;\g\in \Ll_\rho\}$$
denote the set of axis of elements in $\Ll_\rho$, 
so that there is a bijective correspondence $\Aa_\rho\cong \P\Ll_\rho$.

On $\P\Ll_\rho$ we put the graph structure $\ov \g\equiv\ov \eta$ if $\ov \g\neq \ov \eta$ and their axis intersect.
We denote by $\Gg_\r$ this graph and proceed to study its connected components. Let $\frakc\subset \Gg_\rho$ be a connected component with vertex set $V(\frakc)$. We observe that if the component consists of a single vertex $\ov \g$, then the closed geodesic associated to $\g$ is simple. Indeed for each $\eta$ in $\G$, the conjugate $\eta\gamma\eta^{-1}$ belongs to $\Ll_\rho$ and if $\ov{\eta\g\eta^{-1}}\neq \ov \g$ the corresponding axis do not intersect.

Let us assume from now on that $|V(\frakc)|\geq 2$ and let 
$$\G_\frakc=\{\g\in\G|\,\g\text{ stabilizes }\frakc\}$$
and $$\Delta_\fc=\bigcup_{\ov \g\in V(\fc)}\{\g^-,\g^+\}.$$

Then we clearly have that if $\ov \g$ belongs to $V(\frakc)$ then $\g$ is an element of $\G_\frakc$ and $\Delta_\fc$ is a subset of the limit set $\L(\G_\fc)\subset \partial \H^2$ of $\G_\fc$. In particular, since $\D_\fc$ is $\G_\fc$-invariant, we get $\ov \D_\fc=\L(\G_\fc)$.

\begin{lem}\label{lem:9.1}
 There is a point $p_\fc\in \Bb_{V}$ with $b^\pm_\g=p_\fc$ for all $\g$ such that $\ov\g\in V(\fc)$.
\end{lem}
\begin{proof}
 Indeed, if $\ov\g$ is adjacent to $\ov \eta$ we have $b_\g^+=b_\g^-=b_\eta^+=b_\eta^-$ (cfr. Lemma \ref{lem:8.14}), the lemma follows from the assumption that $\fc$ is connected.
\end{proof}
\begin{lem}\label{lem:9.2}
 For every $\g\in\G_\fc$, $\rho(\g)p_\fc=p_\fc$
\end{lem}
\begin{proof}
 For every $\g\in \G_\fc$, if $\eta$ gives a vertex of $V(\fc)$, the same holds for $\g\eta\g^{-1}$. Hence we get 
 $$b_\eta^\pm=b_{\g\eta\g^{-1}}^\pm=\rho(\g)b^\pm_\eta.$$
\end{proof}

\begin{lem}\label{lem:A}
Let $g$ be a geodesic such that $\D_\fc\cap(\!(g^-,g^+)\!)\neq \emptyset$ and $\D_\fc\cap (\!(g^+,g^-)\!)\neq \emptyset$. Then there exists $\ov \g\in\fc$ with $\ax(\g)\cap g\neq\emptyset$.
\end{lem}
\begin{proof}
Let us choose a class $\ov \eta\in\fc$ with $\eta^+\in(\!(g^-,g^+)\!)$ and a class $\ov\tau\in\fc$ with $\tau^-\in (\!(g^+,g^-)\!)$. Since $\fc$ is connected there is a sequence $\ov\alpha_1=\ov \eta,\ov\alpha_2,\ldots,\ov\alpha_n=\ov \tau$ of classes in $\fc$ such that, for every $i$, the axis $\ax(\alpha_i)$ intersects $\ax(\alpha_{i+1})$. But then clearly there is an index $j$ such that $\ax(\alpha_j)$ intersects the geodesic $g$.
\end{proof}

If $X$ is a subset of $\ov \H^2=\H^2\cup \partial\H^2$, we denote by $\ov {{\rm Co}(X)}$ the closed convex hull of $X$ in $ \H^2$.   
To any component $\fc$ we associate the closed convex subset $Y_\fc$ of ${\H^2}$ defined by 
$$Y_\fc=\ov{{\rm Co}(\L(\G_\fc))}=\ov{{\rm Co}(\D_\fc)}.$$
We say that an element $\g\in\G_\fc$ is a boundary component if the axis of $\g$ is a boundary component of $Y_\fc$.

\begin{prop}\label{prop:3}
 For every primitive, hyperbolic element $\g\in \G_\fc$ which is not a boundary component we have
 $$\ov\g\in V(\fc).$$
\end{prop}
\begin{proof}
 Since $\g$ stabilizes $\fc$ and is not a boundary component, we have that the intersection $\D_\fc\cap (\!(\g^-,\g^+)\!)$ is not empty and similarly $\D_\fc\cap (\!(\g^+,\g^-)\!)$ is not empty. Thus we conclude by Lemma \ref{lem:A}.
\end{proof}
Our next aim is to show that the image $p(Y_\fc)$ of $Y_\fc$ under the universal covering map $p:\H^2\to\S$ is a compact subsurface of $\S$ with geodesic boundary.  
\begin{prop}\label{lem:9.4}
Let $\fc\subset \calG_\rho$  be a connected component with more than one vertex. For every $\g\in\G$ one of the following holds:
\begin{enumerate}
\item $\gamma Y_\fc=Y_\fc$,
\item $\gamma Y_\fc\cap Y_\fc$ is a boundary component of $Y_\fc$,
\item the intersection $\gamma Y_\fc\cap Y_\fc$ is empty.
\end{enumerate}
\end{prop}

\begin{proof}
First we show that if the intersection $\gamma \mathring Y_\fc \cap Y_\fc$ is not empty, then $\gamma \fc=\fc$ and hence $\gamma Y_\fc=Y_\fc$. Let $x\in \g\mathring Y_\fc\cap Y_\fc$, and assume by contradiction that $\g\fc\neq \fc$, which  implies that $\g V(\fc)\cap V(\fc)=\emptyset$.
\begin{claim}\label{claim:1}
The point $x$ does not belong to $\ax(\eta)$ for any $\ov\eta\in \fc$.
\end{claim}

Assume, instead, that $x$ belongs to $\ax(\eta)$ for some element $\eta$ with $\ov \eta \in \fc$. If the intersection $\D_{\g\fc}\cap(\!(\eta^-,\eta^+)\!)$ is empty, then $\D_{\g\fc}$ is contained in the closed interval $[[\eta^+,\eta^-]]$ and hence $Y_{\g\fc}$ is contained in one of the closed halfplanes determined by $\ax(\eta)$. This contradicts the hypothesis that $x$ belongs to the interior of $Y_{\g\fc}$. Thus we have that both intersections
$\D_{\g\fc}\cap (\!(\eta^-,\eta^+)\!)$ and $\D_{\g\fc}\cap (\!(\eta^+,\eta^-)\!)$ are not empty. But then, by Lemma \ref{lem:A}, there is an element $\xi\in \g\fc$ whose axis $\ax(\xi)$ intersects $\ax(\eta)$. This implies that either $\ov \xi=\ov \eta$ or the elements $\ov\xi$ and $\ov \eta$ are adjacent in the graph $\calG_\r$. Both contradict the fact that $\g V(\fc)\cap V(\fc)=\emptyset$, and this proves Claim \ref{claim:1}.

Now we can define, for every $\ov g\in\fc$, $B_{\ov g}$ to be the unique closed interval in $S^1$ with endpoints $\{g^-,g^+\}$ and such that $x$ does not belong to the convex hull $\ov{{\rm Co}(B_g)}$. According to Claim \ref{claim:1}, this is well defined.
\begin{claim}\label{claim:2}
For every $\ov g$ in $\fc$, the intersection $\D_{\g\fc}\cap B_{\ov g}$ is empty.
\end{claim}
Indeed, assume that the intersection is not empty for some $\ov g\in\fc$. Since $\ov g$ does not belong to $\g\fc$, this implies that the intersection $\D_{\g\fc}\cap\mathring B_g$ is not empty. Since $x$ belongs to $\g\mathring Y_\fc$ we get  that the intersection $\D_{\g\fc}\cap (S^1\setminus B_g)$ is not empty, and hence, by Lemma \ref{lem:A}, there is $\ov \xi\in \g\fc$ whose axis $\ax(\xi)$ intersects $\ax(g)$ nontrivially. This again contradicts the assumption $\g V (\fc)\cap V(\fc)=\emptyset$.
\begin{claim}\label{claim:3}
The union $\bigcup_{\ov g\in\fc}B_{\ov g}$ is connected.
\end{claim}

Indeed, for any pair of adjacent elements $\ov \g$ and $\ov\eta$ in $\fc$, we have that the intersection $B_{\ov \g}\cap B_{\ov\eta}$ is not empty. Now enumerate $\fc$ by a possibly redundant sequence $\ov \g_1,\ov \g_2,\ldots$ of consecutive adjacent vertices. Then the union $\bigcup_{i=1}^{\infty}B_{\ov\g_i}$ is connected.
\smallskip

Since the union  $\bigcup_{\ov g\in\fc}B_{\ov g}$ is connected, it is an interval of $S^1$ say with endpoints $\alpha_1,\alpha_2$, numbered such that 
$$(\!(\alpha_1,\alpha_2)\!)\subset \bigcup_{\ov g\in\fc}B_{\ov g}\subset [[\alpha_1,\alpha_2]].$$

It follows then from Claim \ref{claim:2} that the intersection $\D_{\g\fc}\cap (\!(\alpha_1,\alpha_2)\!)$ is empty, on the other hand $\D_{\fc}\subseteq \bigcup_{\ov g\in\fc}B_{\ov g}\subset [[\alpha_1,\alpha_2]].$ This implies that $Y_\fc$ and $Y_{\g\fc}$ lie in different half planes determined by the geodesic joining $\alpha_1$ to $\alpha_2$ and hence the intersection $\g\mathring Y_{\fc}\cap Y_{\fc}$ is empty. This gives a contradiction.
\smallskip

Assume now that $\g Y_\fc$ is different from $Y_\fc$ and that the intersection $\g Y_\fc\cap Y_\fc$ is not empty. Let $x$ be a point in the intersection $\g Y_\fc\cap Y_\fc$, then $x$ belongs to the boundary of $\g Y_\fc$ and also to the boundary of $Y_\fc$. Let $g$ and $g'$ be the geodesics giving respectively the connected components of $\partial (\g Y_\fc)$ and $\partial (Y_\fc)$ containing $x$. 

If $g\cap g'=\{x\}$, then the intersection of the interiors $\g \mathring Y_\fc\cap \mathring Y_\fc$ is not empty which, together with what we proved, implies that the $\g\mathring Y_\fc$ is equal to $\mathring Y_\fc$equal and leads to a contradiction. Thus $g=g'\subseteq \partial (\g Y_\fc)\cap \partial Y_\fc$. Since the intersection $\g \mathring Y_\fc\cap \mathring Y_\fc$ is empty, we deduce that $\g Y_\fc$ and $Y_\fc$ lie on different sides of $g$ and hence  $\partial (\g Y_\fc)\cap \partial Y_\fc=g$.
\end{proof}

\begin{prop}\label{prop:9.5}
Let $\fc\subset \Gg_\rho$ be a component with more than one vertex. Let $\G_\fc$ be the stabilizer of $\fc$ in $\G$ and $Y_\fc\subset \H^2$ be the closed convex hull of the limit set of $\G$. Then the map 
$$\G_\fc\backslash Y_\fc\hookrightarrow \G\backslash \H^2$$ 
induces an embedding with image a compact surface with geodesic boundary.
\end{prop}
\begin{proof}
Let us enumerate the vertices  $\{\ov \g_1,\ov \g_2,\ldots\}$ of $V(\fc)$ in such a way that, for each $i$, $\ov \g_i$ is adjacent to $\ov \g_{i+1}$.  Let $\wt x_0$ be the intersection $\ax(\g_1)\cap \ax(\g_2)$ and define $X_n=\bigcup_{i=1}^n\ax(\g_i)$. By construction $X_n$ is connected. Let furthermore $x_0$ denote the projection $x_0=p(\wt x_0)$.

Let $\G_n<\G$ be the image of the natural map $\pi_1(p(X_n),x_0)\to \pi_1(\S,x_0)$ induced by the inclusion $p(X_n)\hookrightarrow \S$. Then $\G_n$ is the fundamental group of the surface $\S_n\subseteq \S$ obtained by taking an appropriate tubular neighborhood of $p(X_n)\subseteq \S$ and adding to it all components of the complement which are either simply connected or whose fundamental group is generated by a parabolic element of $\G$. Then $\S_n$ is a subsurface with smooth boundary and of finite topological type. Since $\G_n< \G_{n+1}$, there exists $N\geq 1$ with $\G_n=\G_N$ for all $n\geq N$. 

We will finish the proof by showing that $\G_\fc=\G_N$. Since $\G_n\wt x_0\subset X_n$, we have $\G_n<\G_{\fc}$.
Conversely, let us take $\g\in \G_\fc$, then $\g\wt x_0=\ax(\g\g_1\g^{-1})\cap \ax(\g\g_2\g^{-1})$ and since $\G_\fc$ preserves $V(\fc)$ we have that $\ov{\g\g_1\g^{-1}}$ and $\ov{\g\g_2\g^{-1}}$ are in $V(\fc)$. Thus $\g\wt x_0\in X_n$ for $n$ large enough which implies $\g\in\G_n$. As a conclusion we get $\G_N=\G_\fc$, which implies that $\G_\fc\backslash Y_\fc$ in $\S$ is isotopic to $\S_N$.
\end{proof}

\begin{proof}[Proof of Theorem \ref{thm:dec}]
The set of isolated components of $\Gg_\rho$ is a $\G$ invariant subset. Since we know that each isolated component of $\Gg_\rho$ corresponds to a geodesic of $\H^2$ that projects to a simple closed curve, we have that the projection of all the isolated components is a collection $\calC$ of pairwise disjoint simple closed curves which cut the surface $\S$ in subsurfaces $\{\S_v\}_{v\in \calV}$ for some index set $\calV$.

Moreover, for any component $\fc$ consisting of more than one element we have that $Y_\fc=\ov{{\rm Co}(\bigcup_{\ov\g\in\fc}\ax(\g))}$ is a subsurface in $\H^2$ which projects to a subsurface of $\S$ whose boundary consists of elements of $\calC$. In particular there exists  $v\in V$ with $p(Y_\fc)=\S_v$.
\end{proof}

\section{Quasi Isometric embeddings}\label{sec:qi}
Let $\rho:\pi_1(\S,x)\to \Sp(V)$ be a representation admitting a maximal framing and $\S=\bigcup_{v\in \calV}\S_v$ be the corresponding decomposition given by Theorem \ref{thm:dec}. We assume, as usual, that $\S$ is equipped with a hyperbolic metric of finite area and denote by $p:\H^2\to \S$  the canonical projection, so that $\S=\G\backslash \H^2$.

As we have seen in Section \ref{sec:dec}, the decomposition of the surface $\S$ comes from a $\G$-invariant decomposition  
$$\H^2=\bigcup_{w\in\wt\calV} S_w$$
into subsurfaces with totally geodesic boundary. The Bass-Serre tree $\calT=(\wt \calV,E)$ can be identified with the incidence tree of the set $\{S_w|\; w\in \wt \calV\}$. Recall that a pair $\{w_1,w_2\}$ forms an edge if the intersection $S_{w_1}\cap S_{w_2}$ is not empty. In this case the intersection corresponds to the axis of an element of $\G$ that acts on the building $\Bb_V$ with zero translation length and determines an isolated component of the graph $\Gg_\rho$. 

Assume now that for every subsurface $\S_v$ we are in the second case of the dichotomy in the decomposition theorem. Then for every $w\in \wt \calV$, the stabilizer $\G_w$ of $w$ in $\G$ has a canonical fixed point $b_w\in \Bb_V^0$ which equals $b^\pm_\g$ for each $\g\in\G_w$.
\begin{thm}\label{thm:qi}
The map 
$$ \begin{array}{ccc}
\wt \calV&\to &\Bb_V^0\\
w&\mapsto&b_w
\end{array}
$$
is a $\G$-equivariant quasiisometry.
\end{thm}

Let $\eta\in\G$ be an element whose corresponding geodesic is not contained in a subsurface. The axis $\ax(\eta)$ determines a sequence $(w_n)_{n\in\Z}$ of vertices in $\Tt$, namely the consecutive sequence of surfaces $S_{w_n}$ crossed by $\ax(\eta)$. This gives a geodesic path in $\Tt$, which is the axis of the isometry of $\Tt$ induced by $\eta$.
\begin{lem}\label{lem:10.1}
Let us assume that the axis $\ax(\eta)$ crosses the surface $S_w$. Let $\g\in\G_w$ be an element which is not boundary parallel and such that $\ax(\g)$ intersects $\ax(\eta)$. Then
$$
b_w=\pi_{\Bb}(\pr_{\Yy_{\eta}}(\phi(\g^+)))=\pi_{\Bb}(\pr_{\Yy_{\eta}}(\phi(\g^-))).
$$  
In particular $b_w$ belongs to $\mathbb Y_\eta$.
\end{lem}
\begin{center}
\begin{tikzpicture}[scale=.8]


\draw (-2,0) to [in=170, out=10] (2,0);
\draw (0,-2) to [in=-80, out=80] (0,2);
\draw (-1.4,-1.4) to [in=-200, out=20] (1.4,-1.4);
\draw (-1.7,1.1) [in=80, out=-80]   to (-1.7,-1.1);
\draw (1.7,1.1) [in=100, out=-100]   to (1.7,-1.1);

\node at (-2.5,0) {$\eta^-$};
\node at (-3.9,0) {$-T$};
\node at (-3.2,0) [rotate=180]{$\rightsquigarrow$};
\node at (2.5,0) {$\eta^+$};
\node at (3.5,0) {$S$};
\node at (3,0) {$\rightsquigarrow$};
\node at (0,-2.5) {$\gamma^-$};
\node at (1,-2.5) {$0$};
\node at (0.5,-2.5) {$\rightsquigarrow$};
\node at (0,2.5) {$\gamma^+$};
\node at (1.2,2.5) {$l_\infty$};
\node at (.6,2.5) {$\rightsquigarrow$};
\node at (-1.8,-1.8) {$\alpha^-$};
\node at (-3.2,-1.8) {$-\Id$};
\node at (-2.5,-1.85)[rotate=180] {$\rightsquigarrow$};
\node at (1.8,-1.8) {$\alpha^+$};
\node at (2.3,-1.9) {$\rightsquigarrow$};
\node at (2.8,-1.8) {$L$};
\begin{scope}[overlay]
\clip (0,0) circle [radius=2];
\clip (-1.7,-1.1)-- (-1.5,1.1) --  (-1.7,2) --(1.7,2)-- (1.7,1.1)--(1.5,-1.1)-- (1.7,-2)--(-1.7,-2)[preaction={fill=gray!50}];
\end{scope}
\begin{scope}[overlay]
\clip (0,0) circle [radius=2];
\clip (-1.5,-1.1)-- (-1.7,1.1) --  (-1.7,2) --(1.7,2)-- (1.5,1.1)--(1.7,-1.1)-- (1.7,-2)--(-1.7,-2)[preaction={fill=gray!50}];
\end{scope}
\filldraw (-2,0) circle [radius=1pt];
\filldraw (2,0) circle [radius=1pt];
\filldraw (0,2) circle [radius=1pt];
\filldraw (0,-2) circle [radius=1pt];
\filldraw (1.4,-1.4) circle [radius=1pt];
\filldraw (-1.4,-1.4) circle [radius=1pt];

\draw (0,0) circle [radius=2];
\draw (-2,0) to [in=170, out=10] (2,0);
\draw (0,-2) to [in=-80, out=80] (0,2);
\draw (-1.4,-1.4) to [in=-200, out=20] (1.4,-1.4);
\draw (-1.7,1.1) [in=80, out=-80]   to (-1.7,-1.1);
\draw (1.7,1.1) [in=100, out=-100]   to (1.7,-1.1);

\node at (-.7,1.3) {$S_w$};

\end{tikzpicture}
\end{center}

%
%
%
%
%
\begin{proof}
Without loss of generality we assume that the 4-tuple $(\eta^-,\g^-,\eta^+,\g^+)$ is positively oriented, we will show that all the eigenvalues of the crossratio $R(\phi(\eta^-),\phi(\g^-),\phi(\eta^+),\phi(\g^+))$ have the form $1+f$ for a positive $f$ satisfying $\|f\|=1$.

Since $\g$ is not boundary parallel we can find $\alpha\in\G_w$ such that $\alpha^-$ belongs to $ (\!(\eta^-,\gamma^-)\!)$, and $\alpha^+$ belongs to $(\!(\g^-,\eta^+)\!)$. Since  $(\alpha^-,\g^-,\alpha^+,\g^+)$ is positively oriented, we can pick an element $g\in \Sp(V)$ with $g_*(\alpha^-,\g^-,\g^+)=(-\Id,0,l_\infty)$. For such $g$ we set $g_*\phi(\eta^-)=-T$,  $g_*\phi(\eta^+)=S$ and $g_*\phi(\alpha^+)=L$. With these notations we have $T>\!>\Id$ and $S>\!>L>\!>0$, moreover the crossratio $R(\phi(\eta^-),\phi(\g^-),\phi(\eta^+),\phi(\g^+))$ is conjugate to $R(-T,0,S,l_\infty)=\Id+ T^{-1}S$.

First observe that all the eigenvalues of $R(-T,0,S,l_\infty)$ are smaller than the corresponding eigenvalues of $R(-\Id,0,S,l_\infty)$. Indeed the first matrix is conjugate to  $\Id+S^{1/2}T^{-1}S^{1/2}$ and the second equals $\Id+S$, moreover all the eigenvalues of $T$ are by assumption greater than 1. Now $\alpha$ and $\g$ cross and have zero translation length since they both belong to $\G_w$. Since $\eta^+$ belongs to $ (\!(\alpha^+,\gamma^+)\!)$, it follows from Lemma \ref{lem:8.12} that all the eigenvalues of $R(\phi(\alpha^-),\phi(\gamma^-),\phi(\eta^+),\phi(\g^+))$  have the form $1+\l$ for a positive $\l$ satisfying $\|\l\|=1$. This implies that for each eigenvalue $\nu_i$ of $\Id+T^{-1}S$ we have $\|\nu_i-1\|\leq 1$.

On the other hand all the eigenvalues of  $R(-T,0,S,l_\infty)$ are bigger than the corresponding eigenvalues of  $R(-T,0,L,l_\infty)$: indeed the first matrix is conjugate to $\Id+T^{-1/2}ST^{-1/2}$ and the second is conjugate to $\Id+T^{-1/2}LT^{-1/2}$. This implies that, denoting by $\mu_i$ the eigenvalues of $R(-T,0,L,l_\infty)$ we have that $\|\nu_i-1\|\geq\|\mu_i-1\|$. This is enough to conclude: we have by Lemma \ref{lem:2.7} that $R(-T,0,L,l_\infty)\cong R(L,l_\infty,-T,0)$ and, as a consequence of Lemma \ref{lem:8.12}, this latter crossratio has all its eigenvalues of the form $1+f$ for some positive $f$ of norm one.

Now we exploit that $b_w$ is in particular equal to $b_\g^\pm$. This latter point is, in view of of Proposition \ref{prop:lazy} equal to $\pi_{\Bb}(\pr_{\Yy_\g}(\phi(\eta^-)))$. Moreover we deduce from Lemma \ref{lem:8.14} that $$\pi_{\Bb}(\pr_{\Yy_\g}(\phi(\eta^-)))
=\pi_{\Bb}(\pr_{\Yy_\eta}(\phi(\gamma^-)))
=\pi_{\Bb}(\pr_{\Yy_\eta}(\phi(\gamma^+)))$$ and this concludes.
\end{proof}

\begin{lem}\label{lem:10.2}
Let $a,b\in\Ll(V)$ be transverse subspaces and fix $x_1,\ldots,x_k\in (\!(a,b)\!)$ such that $(a,x_i,x_{i+1},b)$ is maximal for all $i$. Then
$$\sum_{i=1}^{k-1}{d}(\pr_{\Yy_{a,b}}(x_i),\pr_{\Yy_{a,b}}(x_{i+1}))\leq \sqrt n \,{d}(\pr_{\Yy_{a,b}}(x_1),\pr_{\Yy_{a,b}}(x_{k})).$$
\end{lem}

\begin{proof}
Since for each pair of symmetric matrices $S,T$ we have $\det R(0,S,T,l_\infty)=\det S^{-1}\det T$ we deduce 
$$\det R(a,x_1,x_k,b)=\prod_{j=1}^{k-1}\det R(a,x_j,x_{j+1},b).$$
Thus we get 
$$\ln\|\det R(a,x_1,x_k,b)\|=\sum_{j=1}^{k-1}\ln\|\det R(a,x_j,x_{j+1},b)\|.$$
From Lemma \ref{lem:8.6} we deduce immediately
$$\sum_{i=1}^{k-1}{d}(\pr_{\Yy_{a,b}}(x_i),\pr_{\Yy_{a,b}}(x_{i+1}))\leq \sqrt n\, {d}(\pr_{\Yy_{a,b}}(x_1),\pr_{\Yy_{a,b}}(x_{k})).$$
\end{proof}

\begin{proof}[Proof of Theorem \ref{thm:qi}]
Let $v,w$ be vertices of $\calT$ and pick an element $\eta\in\G$ whose associated axis in $\calT$ contains the geodesic path between $v$ and $w$. Let us name $v_0=v,v_1,\ldots,v_k=w$ the vertices in such path.

We choose, for every $i$ an element $\g_i\in S_{v_i}$ whose axis $\ax(\g_i)$ intersects the axis $\ax(\eta)$ nontrivially, and with the property that $\g_i^{+}\in(\!(\eta^-,\eta^+)\!)$. Then we have that, for every $i$, the 4-tuple $$(\phi(\eta^-),\phi(\g_{i}^+), \phi(\g_{i+1}^+),\phi(\eta^+) )$$ is maximal and hence by Lemma \ref{lem:10.1} and \ref{lem:10.2}  we have 
$$\sum_{i=0}^{k-1} {d_\Bb}(b_{v_i},b_{v_{i+1}})\leq \sqrt n {d_\Bb}(b_{v_0},b_{v_k}).$$

Now, since the number of $\G$-orbits on the set of edges of $\calT$ is finite, there are constants $C_1,C_2$ with 
$$C_1\leq {d_\Bb}(b_l,b_r)\leq C_2$$
for every pair $(l,r)$ of adjacent vertices. Thus we get
$$kC_1\leq\sqrt n {d_\Bb}(b_{v_0},b_{v_k}),$$
which implies $${d}_{\calT}(v_0,v_k)\leq \frac{\sqrt n}{C_1} {d_\Bb}(b_{v_0},b_{v_k}).$$
The inequality
$${d_\Bb}(b_{v_0},b_{v_k})\leq C_2k=C_2{d}_{\calT}(v_0,v_k)$$
is immediate.
\end{proof}
\section{Ultralimits of maximal representations}\label{sec:10}

In this section we apply the general theory developed so far to the field of hyperreals and the Robinson field in order to deduce the decomposition theorem for ultralimits of maximal representations.
\subsection{Hyperreals and Robinson fields}
Let $\omega:\calP(\N)\to \{0,1\}$  be a non-principal ultrafilter on the set of integers. Recall that the ultraproduct $\prod_{\o} X_i$ of a sequence $X_i$, $i\in\N$, of sets is the quotient of $\prod_{i\in\N}X_i$ by the equivalence relation $(x_i)\sim (y_i)$ if $\o(\{i|\; x_i=y_i\})=1$.  We denote by $\l_\o:\prod_{i\in\N}X_i\to\prod_\o X_i$ the quotient map and write $X_\o$ for $\prod_\o X$. In particular $\R_\o$ is the field of hyperreals, and 
if $X_i$ are vector spaces over $\R$, $\R$-algebras, groups then $\prod_\o X$ is a $\R_\o$-vector space, an $\R_\o$-algebra, a group and $\l_\o$ is a morphism in the appropriate category.
For a $\R$-vector space $V$ the map 
$$\begin{array}{cccc}
   &V\times \R_\o &\to& V_\o\\&(v,[(l_i)])&\mapsto&[(l_iv)]
  \end{array}
$$
induces an $\R_\o$-isomorphism $V\otimes_\R \R_\o\mapsto V_\o.$
For $V$ finite dimensional at least, we deduce from the isomorphism
$\End_{\R_\o}(V\otimes_\R \R_{\o})\cong (\End V)\otimes_\R\R_\o$
that the map 
$$\begin{array}{cccc}
   &\prod_{i\in\N} \End(V)&\to& \End(V_\o)\\
   &(T_i)_{i}&\mapsto &T
  \end{array}
$$
where $T([v_i])=[T_i(v_i)]$ induces an algebra isomorphism
$(\End(V))_\o\cong\End (V_\o)$
which restricts to a group isomorphism
$(\GL(V))_\o\cong\GL (V_\o)$.
By abuse of notation we will also denote $
\l_\o:\prod_{\N}\GL(V)\to\GL(V_\o) $
the induced map. Given a symplectic form $\<\cdot,\cdot\>$ on $V$ let $\<\cdot,\cdot\>_\o$ denote the symplectic form on $V_\o$ obtained by extending the scalars from $\R$ to $\R_\o$. Given a sequence of representations $\rho_i:\G\to \Sp(V)$, we will denote by $\rho_\o$ the representation of $\G$ into $\Sp(V_\o)$ obtained by composing $\prod_{i\in\N}\rho_i$ with $\l_\o$.
\begin{prop}\label{prop:11.1}
Assume that $\rho_i:\G\to \Sp(V)$ is a sequence of maximal representations. Then $\rho_\o:\G\to\Sp(V_\o)$ admits a maximal framing.
\end{prop}
The proof uses the following lemma, which is a straightforward verification:
\begin{lem}\label{lem:10.3}
\begin{enumerate}
\item The map $\prod_\N\Gr_k(V)\to\Gr_k(V_\o)$ defined by $(L_i)_{i\in \N}\mapsto \prod_\o L_i$ induces a $(\GL(V))_\o\cong\GL (V_\o)$ equivariant bijection $(\Gr_k(V))_\o\cong \Gr_k(V_\o)$ and restricts to a $(\Sp(V))_\o\cong \Sp(V_\o)$-equivariant bijection $\Ll(V)_\o\cong \Ll(V_\o)$;
\item Let $f_i: W_i\to \R$ be quadratic forms with signature $n_i\in\Z$. Assume that the sequence $\dim W_i$ is bounded and let $f_\o:\prod_\o W_i\to \R_\o$ be the quadratic form given by $f_\o([(v_i)])=[(f_i(v_i))]$. Then $f_\o$ has signature $n$ where $n$ is defined by $\o(\{i|\; n_i=n\})=1$.
\end{enumerate}
\end{lem}

\begin{proof}[Proof of Proposition \ref{prop:11.1}]
Since each $\rho_i$ is maximal, there exists a maximal framing $\phi_i:\partial\H^2\to\Ll(V)$. Define then $\phi_\o:\partial \H^2\to\Ll(V_\o)$ by composing $\prod \phi_i:\partial\H^2\to\prod_\N\Ll(V)$ with the quotient map $\prod_\N\Ll(V)\to \Ll(V_\o)$. The maximality of the so obtained framing follows then from Lemma \ref{lem:10.3} (2).
\end{proof}

Let now $\sigma\in \R_\o$ be an {infinitesimal} and recall the definition of the local ring
$$\Oo_\sigma=\{x\in \R_\o||x|<\sigma^{-k} \text{ for some $k\in \N$}\}$$
with maximal ideal
$$\Ii_\sigma=\{x\in \R_\o||x|<\sigma^{k} \text{ for all $k\in \N$}\}$$
associated to it. The quotient  is the Robinson field 
$\R_{\o,\sigma}=\Oo_\sigma/\Ii_\sigma$
associated to $\s$ \cite{Robinson, LR}.
\begin{remark}
Assuming the continuum hypothesis a deep result of Erd\"os, Gillman and Henriksen \cite{EGH} implies that the field $\R_\o$ does not depend on the choice of the ultrafilter.  And under the same hypothesis Thornton showed that  the normed field $\R_{\o,\s}$ does not depend on the choice of the ultrafilter $\o$ nor on the infinitesimal $\s$ \cite[Theorem 2.34]{Thornton}.

If instead we assume the negation of the continuum hypothesis, it was shown by Kramer, Shelah, Tent and Thomas \cite[Theorem 1.8]{KSTT} that there exists an uncountable set of nonprincipal ultrafilters such that the associated Robinson fields are pairwise non isomorphic.

 If $(\l_i)$ is a divergent sequence of real numbers and we set $\sigma=[(e^{-\l_i})]\in\R_\o$ we have that the field $\R_{\o,\sigma}$ is the field denoted by $\R_{\o,\l}$ in \cite{APcomp}.
\end{remark}
Now let $\rho_\o$ be a representation into $\Sp(V_\o)$ admitting the maximal framing $(S,\phi_\o)$. Choose a compatible complex structure $J_\o$ and an infinitesimal $\s\in\R_\o$ such that $\rho_\o(\G)\subseteq \Sp(V_\o)(\Oo_\s)$, and denote $V_\o$ the vector space $V\otimes_\R\R_\os$.
According to Theorem \ref{thm:bdrymap}, composing $\rho_\o$ with $\pi_\s:\Sp(V_\o)(\Oo_\s)\to \Sp(V_\os)$ we obtain  a representation which admits $q_\s\circ\phi_\o:S\to \Ll(V_\os)$ as maximal framing.

Thus we obtain in particular:
\begin{cor}\label{cor:10.4}
If $(\rho_i)_{i\in\N}:\G\to \Sp(V)$ is a sequence of maximal representations where $V$ is a real symplectic vector space, $\rho_\o:\G\to \Sp(V_\o)$ the corresponding representation over the field of hyperreals, $J_\o$ a choice of compatible complex structure and $\s$ an infinitesimal such that $\rho_\o(\G)\subset \Sp(V_\o)(\Oo_\s)$, then the representation $\rho_\os:\G\to \Sp(V_\os)$ admits a maximal framing defined on $\partial\H^2$.
\end{cor}
In the compact case we obtain a converse:
\begin{thm}\label{thm:10.5}
Assume that the surface $\G\backslash\H^2$ is compact. Then a representation $\rho:\G\to \Sp(V_\os)$ admits a maximal framing if and only if there is a sequence $\rho_i:\G\to\Sp(V)$ of maximal representations such that $\rho_\os=\rho$.
\end{thm}
\begin{proof}
Let 
$${\rm Rep}_g:=\left\{(A_1,B_1,\ldots, A_g,B_g)\in\Sp(V)^{2g}\left|\;\prod_{i=1}^n[A_i,B_i]=\Id\right.\right\}$$
be the $\R$-variety of representations of $\G$ in $\Sp(V)$. Then it follows from \cite{Thornton} that the reduction modulo $\Ii_\s$ induces a surjection ${\rm Rep}_g(\Oo_\s)\to {\rm Rep}_g(\R_\os)$. Thus we can lift $\rho$ to a representation $\rho_\o:\G\to \Sp(V_\o)(\Oo_\s)$ which we represent by a sequence $(\rho_i)_{i\in\N}$ of representations of $\G$ into $\Sp(V)$. 

Let $\phi:\S\to \Ll(V_\os)$ be a maximal framing for $\rho$. It follows from the Collar Lemma that for every hyperbolic element $\g\in\G$, the image $\rho(\g)$ is Shilov hyperbolic. Then $\rho_\o(\g)$ needs also to be Shilov hyperbolic and we have $q_\s(L^+_{\rho_\o(\g)})=L^+_{\rho(\g)}$ because of uniqueness of attractive fixed Lagrangians.

Fix a decomposition of $\S=\G\backslash \H^2$ into pairs of pants, let $P\subseteq\S$ denote any such pair of pants and let $\{c_1,c_2,c_3\}$ be standard generators of $\pi_1(P)$, in particular $c_1c_2c_3=e$. Let $\xi_1,\xi_2,\xi_3$ be the attractive fixed points in $\partial \H^2$ of $c_1, c_2,c_3$. Then $(\xi_1,\xi_2,\xi_3)$ as well as $(\xi_1,c_1\cdot\xi_3,\xi_2)$ are positively oriented. Thus the images under $\phi$ of the two triples are maximal and hence the triples $(L^+_{\rho_\o(c_1)},L^+_{\rho_\o(c_2)},L^+_{\rho_\o(c_3)})$ and $(L^+_{\rho_\o(c_1)},\rho_\o(c_1)L^+_{\rho_\o(c_3)},L^+_{\rho_\o(c_2)})$ are maximal. It follows that there is a set $E_P\subset \N$ of full $\o$-measure such that for each $n$ in $E_P$, $\rho_i(c_1),\rho_i(c_2),\rho_i(c_3)$ are Shilov hyperbolic and both  $(L^+_{\rho_i(c_1)},L^+_{\rho_i(c_2)},L^+_{\rho_i(c_3)})$ and $(L^+_{\rho_i(c_1)},\rho_i(c_1)L^+_{\rho_i(c_3)},L^+_{\rho_i(c_2)})$ are maximal. It follows then from \cite[Theorem 5]{Strubel} that $\rho_i|_{\pi_1(P)}\to \Sp(V)
$  is maximal for each $n$ in $E_P$. Thus if $P_1,\ldots,P_{2g-2}$ is the pair of pants decomposition, we have that for all $i\in\bigcap_{j=1}^{2g-2}E_{p_j}$, the restriction $\rho_i|_{\pi_1(P_i)}$ is maximal. By additivity of the Toledo invariant (see \cite[Theorem 1]{BIW}) we deduce that $\rho_i$ is maximal. Since  $\bigcap_{j=1}^{2g-2}E_{p_j}$ is of full $\o$-measure, this concludes the proof. 
\end{proof}
\subsection{Asymptotic cones}\label{sec:10.2}
We finish the paper deducing the statements about ultralimits of maximal representations from the general theory of representations admitting a maximal framing.
\begin{proof}[Proof of Theorem \ref{thm:1}]
Let $\rho_k:\G\to\Sp(V)$ be a sequence of maximal representations, $J_k\in\mathbb X_V$  a sequence of basepoints, namely a sequence of compatible complex structures, and  $(\l_k)_{k\in\N}$ an adapted sequence of scales. If the sequence $(\l_k)_{k\in\N}$ is bounded on a set of full $\o$-measure, then we may assume 
$$\sup_{k\in\N}\max_{\g\in S} d(\rho_k(\g)J_k,J_k)<\infty$$
and hence, if we conjugate $\rho_k$ by $g_k\in\Sp(V)$ with $g_kJ_k=x$ a fixed basepoint, it follows that the sequence $\pi_k=g_k\rho_kg_k^{-1}$ is relatively compact in the space of representations. In this case $^\o\!\Xx_{\l}$ is just the Siegel space $\Xx_\R$ with rescaled distance and $^\o\rho_\l$ is an ordinary accumulation point of the sequence $(\pi_k)_{k\in\N}$.

If the sequence $(\l_k)_{k\in\N}$ is unbounded, let $\s:=(e^{-\l_k}) $, which is an infinitesimal in $\R_\o$, and let $J_\o:=[(J_k)]\in \End( V_\o)$ which is a compatible complex structure. Then we conclude from the fact that $(\l_k)$ is adapted to $(\rho_k,J_k)$ that $\rho_\o(\G)\subset \Sp(V_\o)(\Oo_\s)$. 

Furthermore it follows from \cite{APcomp} that the action on the Bruhat-Tits building of $\Sp(V_{\os})$ coming from the representation $\rho_\os:\G\to \Sp(V_\os)$ coincides with the ultralimit $^\o\!\rho_\l:\G\to {\rm Iso}(^\o\!\Xx_{\l})$ under the identification of $^\o\!\Xx_{\l}$ with the Bruhat-Tits building $\Bb_{V_\os}$. Theorem \ref{thm:1} follows then from Corollary \ref{cor:10.4} and Theorem \ref{thm:dec}
\end{proof}

 We now characterize the cases which lead to actions without a global fixed point. Recall from the introduction that when $S$ be a finite generating set for $\G$, and $\rho$ is a maximal representation we denote by $D_S(\rho)(x)$ the displacement function. 
 
 The function $D_S(\rho)$ is convex and since $\rho(\G)$ is not contained in any proper parabolic subgroup of $\Sp(V)$ we have that for every $C>0$, the convex set $\{x|\, D_S(\rho)(x)\leq C\}$ must be compact, in particular $D_S(\rho)(x)$  achieves its minimum that we will denote by $\mu_S(\rho)=\min_{x\in\Xx}D_S(\rho)(x)$. 

The function $\rho\mapsto \mu_S(\rho)$ descends then to a proper function 
$${\rm Hom}_{max}(\G,\Sp(V))/\Sp(V)\to [0,\infty)$$
on the character variety of maximal representations. Let now $(\rho_k)_{k\in\N}$ be a sequence of maximal representations, $x_k\in\Xx$ a sequence of basepoints and $\l_k$ an adapted sequence of scales. Furthermore let $y_k\in\Xx$ be such that $\mu_S(\rho_k)=D_S(\rho_k) (y_k)$.

\begin{prop}\label{prop:int2}
 The representation $^\o\rho_\l$ on $^\o\Xx_\l$ has no global fixed point if and only if 
 $$\lim_\o \frac{\l_k}{\mu_S(\rho_k)}<\infty\phantom{XX}\text{ and }\phantom{XX}\lim_\o \frac{d(y_k,x_k)}{\l_k}<\infty$$
 in which case $^\o\!\Xx_\l=^\o\!\Xx_\mu$, the distances on the asymptotic cones are homothetic and the actions $^\o\!\rho_\l$ and $^\o\!\rho_\mu$ coincide.
\end{prop}

\begin{rem}
The fact that if $^\o\!\rho_\l$ has no global fixed point then the limit $\lim_\o \frac{\l_k}{\mu_S(\rho_k)}$ is finite can also be deduced combining \cite[Proposition 4.4]{APcomp} and \cite[Corollary 3]{APell}.
\end{rem}

\begin{proof}[{Proof of Proposition \ref{prop:int2}}]
 For the if part: changing the sequence on a set of $\o$-measure zero, we may assume that for some constant $C>0$, we have $\mu_S(\rho_k)/C\leq \l_k\leq C \mu_S(\rho_k)$ and $d(y_k,x_k)\leq C \l_k$ for all $k\in\N$. This readily implies that the asymptotic cones $^\o\!\Xx_\l$ and $^\o\!\!\Xx_\mu$ are equal, that the induced distances are homothetic with factor $\lim_\o \frac {\l_k}{\mu_S(\rho_k)}$ and that the actions $^\o\!\rho_\l$, $^\o\!\rho_\mu$ coincide. Thus we have to verify that $^\o\!\rho_\mu$ does not have a global fixed point. But this follows immediately from the fact that 
 $$\max_{\g\in S}\frac{d(\rho_k(\g)x,x)}{\mu_S(\rho_k)}\geq 1 \quad \forall x\in\Xx.$$
 
 We next show the only if part. Let $T$ be a finite connected generating set, and let us denote by $K$ the maximal length of an element of $T$ with respect to the generating set $S$. Since $^\o\!\rho_\l$ does not have a global fixed point, it follows from Corollary \ref{cor:3} that  there is $\g_0\in T$ with $L(^\o\rho_\l(\g_0))=\lim_\o\frac {L(\rho_k(\g_0))}{\l_k}>0$. 
Since 
 $$L(\rho_k(\g_0))\leq d(\rho_k(\g_0)y_k,y_k)\leq K\mu_S(\rho_k)\leq K D_S(\rho_k)(x_k)$$ and $\lim_\o \frac {D_S(\rho_k)(x_k)}{\l_k}<\infty$, we may assume that the sequences $(\l_k)_{k\in\N}$ and $(\mu_S(\rho_k))_{k\in\N}$ are equivalent, namely that
 there are positive constants $C_1,C_2$ such that $C_1\mu_S(\rho_k)\leq \l_k\leq C_2\mu_S(\rho_k)$ for all $k\in\N$.

 Pick now two hyperbolic elements $\g,\eta$ in $\G$ with intersecting axes. If $\phi_k:S^1\to\Ll(V)$ denotes the boundary map associated to $\rho_k$, we have $\Yy_{\phi_k(\g^+),\phi_k(\g^-)}\cap \Yy_{\phi_k(\eta^+),\phi_k(\eta^-)}=\{z_k\}$ and the sequence $(z_k)_{k\in\N}$ in $^\o\Xx_\l$ represents a point in the intersection $\mathbb Y_\g^\l\cap \mathbb Y_\eta ^\l$ (see Section \ref{sec:fixedpoint}). Thus we get $\lim_\o \frac {d(x_k,z_k)}{\l_k}<\infty.$
The same applies to $^\o\Xx_\mu$ and hence $\lim_\o \frac {d(y_k,z_k)}{\mu_S(\rho_k)}<\infty.$ Using the triangle inequality and taking into account that the sequences $(\l_k)_{k\in\N}$ and $(\mu_S(\rho_k))_{k\in\N}$ are equivalent, we deduce 
$$\lim_\o \frac {d(x_k,y_k)}{\l_k}<\infty.$$
\end{proof}

\begin{proof}[{Proof of Corollary \ref{cor:intro1}}]
 The first inequality follows from the Collar Lemma, while the last follows by contradiction from Proposition \ref{prop:int2}.
\end{proof}

\begin{proof}[Proof of Corollary \ref{cor:intro2}]

Applying iteratively Theorem \ref{thm:1} it is possible to obtain a canonical decomposition of the surface in subsurfaces with geodesic boundary with the property that  all curves strictly contained in a subsurface have the same growth rate. The set $\calC$ of curves defining this decomposition is the union of the curves given by Theorem \ref{thm:1} and all the curves contained in subsurfaces of type (B) selected by applying Theorem \ref{thm:1} to the restrictions of the representations to those subsurfaces. One can apply Theorem \ref{thm:1} at most $3g-3+p$ times corresponding to the case when at each step  precisely one curve is added and all the complementary pieces are of type (B). Hence there are at most $3g-3+p$ distinct growth rates among curves having nontrivial intersection with $\mathcal C$. There are three possibilities for the remaining curves: either a curve is contained in a subsurface defined by the decomposition $\calC$, or it is one of the curves in $\calC$ or it corresponds to a puncture in the surface. The claim follows since there are at most $2g-2+p$ complementary components.

\end{proof}

\section{Appendix A (by Thomas Huber)}
\begin{prop}\label{prop:appendix}
Let $\F$ be a real closed field. Let $n$ be a positive integer and assume that $a_1,\ldots,a_n\geq 1$. Then we have
$$(a_1a_2\ldots a_n-1)^n\geq(a_1^n-1)(a_2^n-1)\ldots(a_n^n-1)$$
with equality if and only if $a_1=\ldots=a_n$
\end{prop}
For $\F=\R$ this follows easily from the convexity of the function $\frac{e^x}{e^x-1}$; here we reproduce the proof due to Thomas Huber for general real closed fields.
We start with a key lemma:
\begin{lem}
Let $n$ be a positive integer and let $c,x\geq 1$. Then we have 
\begin{equation}\label{eqn:4}(cx-1)^n\geq (c^nx-1)(x-1)^{n-1}\end{equation}
with equality if and only if $n=1$ or $c=1$.
\end{lem}
\begin{proof}
We use induction. For $n=1$ the inequality is in fact an equality. By induction 
$$(cx-1)^{n+1}=(cx-1)(cx-1)^n\geq (cx-1)(c^nx-1)(x-1)^{n-1}$$
(observe that all factors are non-negative) and it suffices to show that
$$(cx-1)(c^nx-1)\geq(c^{n+1}x-1)(x-1)$$
holds. But the difference of the left and the right hand side factors as 
$$x(c-1)^2(c^{n-1}+\ldots+c+1)$$
and is clearly non-negative.
\end{proof}
Now we turn to the proof of the main result and proceed again by induction. For $n=1$ there is nothing to show, hence let $n\geq 2$. By symmetry we may assume that $a=a_1\geq a_i$ for all $i\geq 2$. By the induction hypothesis the right hand side of the inequality does not decrease when we replace $a_2,\ldots, a_n$ by their geometric mean $b=(a_2\ldots a_n)^{1/(n-1)}$. Therefore it suffices to show the inequality 
$$(ab^{n-1}-1)^n\geq(a^n-1)(b^n-1)^{n-1}$$
where $a\geq b\geq 1$. But this is a direct consequence of our lemma: just set $c=a/b\geq 1$ and $x=b^n\geq 1$ in \ref{eqn:4}. Equality only holds for $c=1$, that is for $a=b$. But this implies $a_1=\ldots=a_n$ by the maximal choice of $a_1$.

\end{document}